\documentclass[10pt,a4paper]{article}

\usepackage[english]{babel}
\usepackage[T1]{fontenc}
\usepackage{bbm}
\usepackage{graphicx}
\usepackage{verbatim}
\usepackage{amsmath}
\usepackage{amssymb}
\usepackage{amsthm}
  \theoremstyle{plain}
  \newtheorem{theorem}{Theorem}
  \newtheorem{lemma}{Lemma}

  \newtheorem{assumption}{Assumption}
  \newtheorem{definition}{Definition}
  \newtheorem{example}{Example}
  \theoremstyle{remark}
  \newtheorem{remark}{Remark}
\usepackage{mathrsfs}
\usepackage{mathtools}
\usepackage[left=2.7cm,right=2.7cm,top=3.1cm,bottom=3.1cm]{geometry}
\usepackage[colorlinks=true, allcolors=blue, hypertexnames=false]{hyperref}

\usepackage{algpseudocode,float,algorithmicx,algorithm}
\makeatletter
\newenvironment{breakablealgorithm}
{% \begin{breakablealgorithm}
	\begin{center}
		\refstepcounter{algorithm}% New algorithm
		\hrule height.8pt depth0pt \kern2pt% \@fs@pre for \@fs@ruled
		\renewcommand{\caption}[2][\relax]{% Make a new \caption
			{\raggedright\textbf{\ALG@name~\thealgorithm} ##2\par}%
			\ifx\relax##1\relax % #1 is \relax
			\addcontentsline{loa}{algorithm}{\protect\numberline{\thealgorithm}##2}%
			\else % #1 is not \relax
			\addcontentsline{loa}{algorithm}{\protect\numberline{\thealgorithm}##1}%
			\fi
			\kern2pt\hrule\kern2pt
		}
	}{% \end{breakablealgorithm}
		\kern2pt\hrule\relax% \@fs@post for \@fs@ruled
	\end{center}
}
\makeatother

 \allowdisplaybreaks[4]

\begin{document}

\title{Accelerated Optimization Landscape of Linear-Quadratic Regulator \thanks{This work is supported in part by the National Natural Science Foundation of China (62173191).}}
\author{Lechen Feng\thanks{College of Artificial Intelligence, Nankai University, Tianjin 300350, P.R. China. Email: {\tt fenglechen0326@163.com}.}  \and Yuan-Hua Ni\thanks{College of Artificial Intelligence, Nankai University, Tianjin 300350, P.R. China. Email: {\tt yhni@nankai.edu.cn}.}}
	
\date{\today}

\maketitle

\begin{abstract}
Linear-quadratic regulator (LQR) is a landmark problem in the field of optimal control, which is the concern of this paper. Generally, LQR is classified into state-feedback LQR (SLQR) and output-feedback LQR (OLQR) based on whether the full state is obtained.
  It has been suggested in existing literature that both  SLQR and  OLQR could be viewed as \textit{constrained nonconvex matrix optimization} problems in which the only variable to be optimized is the feedback gain matrix.
 In this paper, we introduce for the first time an accelerated optimization framework of handling the LQR problem, and give its convergence analysis for the cases of SLQR and OLQR, respectively.

  Specifically, a Lipschitz Hessian property of LQR performance criterion is presented, which turns out to be a crucial property for the application of modern optimization techniques.
  For the SLQR problem, a continuous-time hybrid dynamic system is introduced, whose solution trajectory is shown to converge exponentially to the optimal feedback gain with Nesterov-optimal order $1-\frac{1}{\sqrt{\kappa}}$ ($\kappa$  the condition number).
  Then, the symplectic Euler scheme is utilized to discretize the hybrid dynamic system, and a Nesterov-type method with a restarting rule is proposed that preserves the continuous-time convergence rate, i.e., the discretized algorithm admits the Nesterov-optimal convergence order.
  For the OLQR problem, a Hessian-free accelerated framework is proposed, which is a two-procedure method consisting of semiconvex function optimization and negative curvature exploitation.
  In a time $\mathcal{O}(\epsilon^{-7/4}\log(1/\epsilon))$, the method can find an $\epsilon$-stationary point of the performance criterion; this entails that the method improves upon the $\mathcal{O}(\epsilon^{-2})$ complexity of vanilla gradient descent. Moreover, our method provides the second-order guarantee of stationary point.

\textbf{Keywords:} linear-quadratic regulator, accelerated optimization, hybrid dynamic system, symplectic Euler discretization

\end{abstract}

\section{Introduction}

%1段， 为何用优化视角来研究LQR，有哪些文献研究了， [][][], 为机器学习中的策略梯度提供。。。。。，不是为了和ARE比较

%2 段，优化视角下还有加速算法的研究必要性， 讲讲加速算法的必要性和重要性

Direct policy gradient optimization is a hot research topic of dynamic decision problems that include optimal control as a particular case; for example, see    \cite{z9,z10,z12,z14,z15,ref1,ref2,ref4} and references therein that relate to the topic of this paper.
The reasons of studying direct policy gradient optimization lie in at least the following points.
First,  policy gradient is a ``end-to-end'' method that allows us to optimize the performance criterion directly.
Second, direct policy gradient inherently enables or inherits extensively parameterized policies, allowing the policies to be explicitly modeled in a variety of ways a priori.
Third, the solvability of model-free control and decentralized control problems may be shed light on by the policy gradient method.
%
%
%Following the policy gradient optimization, this paper intends to develop

It is imperative that accelerated optimization methods should be brought up whenever the gradient descent is broached.
In particular, Heavy-ball method \cite{i1} and Nesterov's acceleration gradient (NAG) method \cite{i2}, two representatives of accelerated optimization algorithms, have played significant roles in the disciplines of economics, aviation, and information sciences.
Regretfully, accelerated algorithms have not been developed in the area of policy optimization, especially the area of LQR policy optimization; this brings a significant gap between areas of pure optimization and policy optimization.
In this paper, we present a first-order accelerated policy optimization framework for LQR problems based on modern optimization techniques and continuous-time optimization theory.
In contrast to identifying an efficient algorithm for solving the Algebraic Riccati Equation (ARE) (that is generally a second-order method), such an accelerated framework is actually compared to the standard gradient descent method proposed in \cite{ref2}.

\subsection{Related Works}

{\textbf{Continuous-time optimization:}}
Continuous-time optimization algorithm, which links optimization problem with dynamic system and cybernetics closely, has been one of hot research topics in the areas of optimization and beyond,
and continuous-time versions of some classic optimization algorithms have been derived; see for example gradient descent \cite{Smith-1993}, Nemirovski's mirror descent \cite{Krichene-2015}, Newton method \cite{Attouch-2001}, Heavy-ball method \cite{Alvarez-SICON2000} and NAG \cite{z1}. In particular, the continuous-time viewpoint of handling NAG and its variants has gained much attention \cite{z3,z4,z5,z6,z7,ref10,ref11} recently.
However, in reality, our primary concern is still the discrete-time algorithm, as computer can only process discrete-time signals due to their physical limitations.
Different discretization schemes are studied for continuous-time NAG-type algorithms such as the three-series method \cite{z3,z4}, different equation solver method \cite{z5,z6}, semi-implicit Euler method \cite{z7,ref11}, symplectic conservation method \cite{z10,Betancourt-2018,Franca-2020,Jordan-NIPS2019}.
Sadly, the question that which discretization scheme is optimal remains open to the authors' best knowledge.

\noindent{\textbf{State feedback LQR:}}
According to whether the full state is measured, LQR, a key issue in the field of optimal control, can be classified into SLQR (state feedback LQR) and OLQR (output feedback LQR).
%Direct policy gradient is a hot topic of SLQR problems in recent.
From the viewpoint of optimization, the SLQR cost is nonconvex but satisfies the Polyak-Lojasiewicz (PL) condition \cite{z9,ref2}.
Hence, gradient-based policy optimization algorithm can be ensured the convergence to the global optimal linear feedback gain
exponentially in both model-based \cite{ref2, Bu-2019} and model-free settings \cite{z10,ref1,Fazel-2018}.
Moreover, policy optimization methods are also studied for SLQR with additional constraints, e.g., decentralized LQR problem \cite{Zheng-2019,Sun-2023}, risk constrained LQR problem \cite{Basar-2023} and $\mathcal{H}_2$/$\mathcal{H}_\infty$ problem \cite{Basar-2019}.

\noindent{\textbf{Output feedback LQR:}}
Due to the absence of PL condition \cite{z12} and the disconnection of feasible set \cite{z13}, the policy optimization of OLQR is challenging and remains some open problems \cite{ref2}.
For static OLQR controller, fortunately, gradient descent can converge to a stationary point \cite{ref2} and to a local minimizer linearly with close  starting point \cite{z14}.
Since static output feedback policies are typically insufficient for achieving good closed-loop control performance, dynamic output feedback policies are investigated in \cite{z15}, where the unique observable stationary point is characterized providing a certificate of optimality for policy gradient methods under moderate assumptions.

In this paper, we examine the model-based continuous-time LQR problem from the viewpoint of optimization and present an  accelerated gradient-based policy optimization framework. If compared with existing results, the main contributions of the presented paper are as follows.
\begin{enumerate}
	
  \item We prove for the first time the Lipschitz Hessian property of performance criterion (with respect to feedback gain), which is essential for the convergence analysis of the algorithms of the  presented paper.

  It should be noted that although an analogous result has been presented for the discrete-time LQR problem \cite{z14}, the methodology of  \cite{z14} cannot be applied to the continuous-time setting since its high dependence on the structure of discrete-time system. Hence, the proof of Lipschitz Hessian property of this paper is of independent interest.

  \item For the SLQR problem and motivated by \cite{ref10}, we propose a gradient-based method with momentum term and restarting rule in both continuous time and discrete time, respectively, for the first time. It is first shown that the gradient-based algorithms  can converge to the optimal feedback gain exponentially with Nesterov-optimal order $1-\frac{1}{\sqrt{\kappa}}=1-\sqrt{\frac{\mu}{L}}$.
      Noting that the PL condition and $L$-smoothness of  \cite{ref10} are globally posed on $\mathbb{R}^n$, instead, the SLQR cost possesses the PL condition and $L$-smoothness only on a bounded sublevel set $\mathcal{S}_0$; hence, a new restarting rule is introduced in this paper to restrict the iteration sequences within $\mathcal{S}_0$, and the parameter selection in the algorithms must be carefully re-examined after each restarting.

  \item For the OLQR problem, we discuss how to find an $\epsilon$-stationary point of OLQR cost with second-order guarantee for the first time. Based on \cite{ref13}, a Hessian-free accelerated framework is introduced and such second-order stationary point can be located in a time $\mathcal{O}(\epsilon^{-7/4}\log(1/\epsilon))$.

      In contrast to \cite{ref13}, the OLQR cost, similar to the SLQR setting, possesses the $L$-smoothness only on a sublevel set rather than on the  whole space $\mathbb{R}^n$. Consequently, we must design a restarting rule and re-evaluate the complexity of the proposed algorithms.
      Noting that the second-order stationary point has been investigated in \cite{ZHeng-2022CDC} for dynamic-feedback LQG problem, the static-feedback OLQR problem of this paper  exhibits totally different structure and requires an independent research.

\end{enumerate}

{\textit{Notation}}. Let $\Vert \cdot\Vert$ and $\Vert \cdot\Vert_F$ be the spectral norm and Frobenius norm of a matrix. $\mathcal{S}^n_+$ is the set of positive definite matrices of $n\times n$; $I$ is identity matrix.  $A\succ B(A\succeq B)$ indicates that the matrix $A-B$ is positive (semi-)definite, whereas $A\prec B(A\preceq B)$ indicates that the it is negative (semi-)definite. Eigenvalues $\lambda_i(A)~i=1,2,\dots,n$ of a matrix $A$ are indexed according to their real parts in the ascending order, i.e., $\Re(\lambda_1(A))\leq\dots\leq\Re(\lambda_n(A))$. We use the notation $u_1\succ u_2$ if the function $u_1$ dominates the function $u_2$ for large arguments, i.e., $\lim_{\kappa\to\infty}u_1(\kappa)/u_2(\kappa)=\infty$, where $u_1,u_2$ are real-valued functions that are positive for large arguments. Similarly, the notation $u_1\prec u_2$ implies that the function $u_2$ dominates $u_1$ for large arguments, and the $u_1\sim u_2$ implies that neither $u_1$ nor $u_2$ are dominant for large arguments; that is,
$$\lim_{\kappa\to\infty}\frac{u_1(\kappa)}{u_2(\kappa)}<\infty,~~\lim_{\kappa\to\infty}\frac{u_2(\kappa)}{u_1(\kappa)}<\infty.$$
\section{Problem Formulation}

Consider a linear time-invariant system
\begin{equation}\label{eq1}
  \begin{aligned}
  &\dot{x}(t)=Ax(t)+Bu(t),\\
  &y(t)=Cx(t),\\
  &\mathbb{E}x(0)x(0)^{\rm T}=\Sigma
  \end{aligned}
\end{equation}
with state $x$, input $u$, output $y$ and matrices $A\in\mathbb{R}^{n\times n},B\in\mathbb{R}^{n\times m},C\in\mathbb{R}^{r\times n}$. The infinite-horizon LQR problem of this paper is to find a linear
feedback gain $K\in\mathbb{R}^{m\times r}$ such that $u(t)=-Ky(t)$  minimizes the following performance criterion
\begin{equation}\label{euq-2}
  f=\mathbb{E}_{x_0}\int_{0}^{\infty}\left(x(t)^{\rm T}Qx(t)+u(t)^{\rm T}Ru(t)\right){\rm d}t.
\end{equation}
The expectation is calculated over the distribution of initial condition $x(0)$ with zero mean and covariance matrix $\Sigma$ given by (\ref{eq1}), and the cost function $f$ is parameterized by $0\prec Q\in \mathcal{S}_+^{n}$ and $0\prec R\in \mathcal{S}_+^{m}$. Moreover, we distinguish the LQR problem via the selection of $C$ :
\begin{itemize}
  \item SLQR : if $C=I$, i.e. the state $x(t)$ is available to control input.
  \item OLQR : if $C\neq I$, the output $y(t)$ is the only information available.
\end{itemize}

Under the feedback control $u(t)=-Ky(t)$, the closed loop system is given by
\begin{equation}\label{eq3}
  \dot{x}(t)=A_Kx(t),~~A_K=(A-BKC),
\end{equation}
and LQR performance becomes
\begin{equation}\label{eq4}
  f(K)=\mathbb{E}_{x_0}\int_{0}^{\infty}\left(x(t)^{\rm T}(Q+C^{\rm T}K^{\rm T}RKC)x(t)\right){\rm d}t.
\end{equation}
Specifically, we use the notation $f(K)$ to highlight that the index of LQR performance depends on gain matrix $K$ only, i.e., finding the optimal control of system (\ref{eq1}) is equivalent to finding the optimal gain
matrix $K$. Thus, our optimization problem is
\begin{equation}\label{eq5}
\min_{K\in\mathcal{S}} f(K),
\end{equation}
where $\mathcal{S}$ is the feasible set defined as the set of stabilizing feedback gains
\begin{equation}\label{eq6}
\mathcal{S}=\left\{K\in\mathbb{R}^{m\times r}:\Re[\lambda_i(A-BKC)]<0,\forall i=1,\dots,n\right\}.
\end{equation}
In particular, our results require an initial stabilizing controller $K_0\in\mathcal{S}$. This assumption is not strict, as it is well acknowledged that stabilizing policies can be achieved via policy gradient approaches \cite{llka-2023}. As a conclusion, we suppose the following assumptions hold.

\begin{assumption}\label{assumption1}
$K_0\in \mathcal{S}$ exists and is available; weighting matrices are positive definite, i.e., matrix $Q,R,\Sigma\succ 0$
\end{assumption}

Note that the integral  (\ref{eq4}) must be evaluated in order to perform the optimization problem (\ref{eq5}).
The next lemma formulates problem (\ref{eq5}) as a matrix nonconvex and constrained optimization problem.
\begin{lemma}[\cite{ref2}]
Optimization problem (\ref{eq5}) is equivalent to the following matrix constrained optimization problem:
\begin{equation}
\begin{aligned}\label{eq7}
\min_K~~&f(K)={\rm Tr}(X\Sigma),\\
{\rm s.t.}~~&(A-BKC)^{\rm T}X+X(A-BKC)+C^{\rm T}K^{\rm T}RKC+Q=0,\\
&X\succ0.
\end{aligned}
\end{equation}
\end{lemma}

\section{Properties of $f(K)$ and $\mathcal{S}$}\label{section3}

%This section presents properties of the objective function $f(K)$ and its feasible set $\mathcal{S}$.

\begin{lemma}[\cite{ref2}]\label{lemma2}
When $C=I$, then the sets $\mathcal{S}$ and $\mathcal{S}_0$ are connected for every $K_0\in\mathcal{S}$, where $\mathcal{S}_0$ is the sublevel set of $\mathcal{S}$, i.e.,
\begin{equation*}
\mathcal{S}_0=\{K\in\mathcal{S}:f(K)\leq f(K_0)\}.
\end{equation*}
\end{lemma}

As local search methods,  gradient-based optimization algorithms typically cannot jump between different connected components of feasible set for seeking the optimal solution \cite{ref4}. Consequently, for gradient-based optimization algorithms, the connectedness of feasible set is crucial. Unfortunately, the connectedness mentioned in Lemma \ref{lemma2} will no longer hold for the OLQR issue. Under some additional presumptions, the number of components will grow exponentially with the expansion of state space's dimension \cite{ref3}.

\begin{definition}
We call a function $f(K)$
  \begin{enumerate}
    \item is coercive on $\mathcal{S}$, if for any sequence $\{K_j\}_{j=1}^\infty\subset\mathcal{S}$ we have
    $$f(K_j)\to +\infty$$
    if $\Vert K_j\Vert\to +\infty$ or $K_j\to K\in\partial\mathcal{S}$.
    \item satisfies the PL condition, if it is continuously differentiable and satisfies
    $$\mu(f(K)-f(K^*))\leq\frac{1}{2}\Vert \nabla f(K)\Vert_F^2,\quad\forall K\in\mathcal{S},$$
    where $\mu$ is some positive constant, and $K^*$ is an optimal solution of $f(K)$  over $\mathcal{S}$.
    \item is $L$-smooth, if its gradient satisfies Lipschitz condition with constant $L$.
  \end{enumerate}
\end{definition}
\begin{theorem}
  The LQR cost $f(K)$ is real analytical and hence twice continuously differentiable. Moreover, the gradient of $f(K)$ is
  \begin{equation*}
    \nabla f(K)=2(RKC-B^{\rm T}X)YC^{\rm T},
  \end{equation*}
  where $Y$ is the solution to the Lyapunov matrix equation
  \begin{equation*}
    A_KY+YA_K^{\rm T}+\Sigma=0.
  \end{equation*}
  The Hessian of $f(K)$ is
  \begin{equation*}
    \frac{1}{2}\nabla^2 f(K)[E,E]=\langle(REC-B^{\rm T}X')YC^{\rm T},E\rangle+\langle MY'C^{\rm T},E\rangle,
  \end{equation*}
  where $M=RKC-B^{\rm T}X$, and $X',Y'$ are the solutions to the following Lyapunov matrix equations
  \begin{align*}
  &A_K^{\rm T}X'+X'A_K+M^{\rm T}EC+(M^{\rm T}EC)^{\rm T}=0,\\
  &A_KY'+Y'A_K^{\rm T}-(BECY+(BECY)^{\rm T})=0.
  \end{align*}
\end{theorem}
\begin{proof}
The gradient and Hessian of $f(K)$ can be found in \cite{ref2}, and here we only provide the proof of real analytical property. First, upon vectorizing the Lyapunov equation
$$A_K^{\rm T}X+XA_K+C^{\rm T}K^{\rm T}RKC+Q=0$$
with $A_K=A-BKC$, we have
\begin{equation*}
(I_n\otimes A_K^{\rm T}){\rm vec}(X)+(A_K^{\rm T}\otimes I_n){\rm vec(X)}=-{\rm vec}(C^{\rm T}K^{\rm T}RKC+Q).
\end{equation*}
Since $A_K$ is stable, we know that $I_n\otimes A_K^{\rm T}+A_K^{\rm T}\otimes I_n$ is invertible, and thus we have
$${\rm vec}(X)=-(I_n\otimes A_K^{\rm T}+A_K^{\rm T}\otimes I_n)^{-1}{\rm vec}(C^{\rm T}K^{\rm T}RKC+Q).$$
It is obvious that each element of $(I_n\otimes A_K^{\rm T}+A_K^{\rm T}\otimes I_n)^{-1}$ is a rational function of element $K$. Therefore, the objective function $f(K)={\rm Tr}(X\Sigma)$ is a rational
function of element $K$, which is real analytical indeed.
\end{proof}

The following theorem adopted from \cite{ref2} is on the $L$-smoothness and PL condition of $f(K)$, which is rewritten here for the convenience of subsequent discussions.

\begin{theorem}[\cite{ref2}]
	
	The following assertions hold.
	
  \begin{enumerate}
    \item Function $f(K)={\rm Tr}(X\Sigma)$ is coercive and the following estimates hold
    \begin{align*}
    f(K)&\geq \frac{\lambda_1(\Sigma)\lambda_1(Q)}{-2\Re\lambda_n(A_K)},\\
    f(K)&\geq \frac{\lambda_1(\Sigma)\lambda_1(R)\Vert K\Vert_F^2\lambda_1(CC^{\rm T})}{2\Vert A\Vert+2\Vert K\Vert_F\Vert B\Vert\Vert C\Vert}.
    \end{align*}
    \item On the sublevel set $\mathcal{S}_0$, function $f(K)$ is $L(f(K_0))$-smooth with constant
    \begin{equation*}
    \begin{aligned}
      L(f(K_0))=&\dfrac{2f(K_0)}{\lambda_1(Q)}(\lambda_n(R)\Vert C\Vert^2+\Vert B\Vert\Vert C\Vert_F\xi(f(K_0)))
      \end{aligned}
    \end{equation*}
    related to the initial feedback gain matrix $K_0$, where
    \begin{equation*}
      \begin{split}
         \xi(f(K_0))=\frac{\sqrt{n}f(K_0)}{\lambda_1(\Sigma)}\left(\frac{f(K_0)\Vert B\Vert}{\lambda_1(\Sigma)\lambda_1(Q)}+\sqrt{\left(\frac{f(K_0)\Vert B\Vert}{\lambda_1(\Sigma)\lambda_1(Q)}\right)^2+\lambda_n(R)}\right).
      \end{split}
    \end{equation*}
    \item When $C=I$, function $f(K)$ satisfies the PL condition on the set $\mathcal{S}_0$
    \begin{equation*}
    \frac{1}{2}\Vert \nabla f(K)\Vert_F^2\geq\mu(f(K_0))(f(K)-f(K^*)),
    \end{equation*}
    where $\mu(f(K_0))>0$ is a constant related to the initial feedback gain matrix $K_0$ given by
    $$\mu(f(K_0))=\dfrac{\lambda_1(R)\lambda_1^2(\Sigma)\lambda_1(Q)}{8f(K^*)\left(\Vert A\Vert+\frac{\Vert B\Vert^2f(K_0)}{\lambda_1(\Sigma)\lambda_1(R)}\right)^2}.$$
  \end{enumerate}
\end{theorem}
\begin{theorem}[$M$-Lipschitz Hessian]\label{LipschizHessian}
  For any $K,K'\in\mathcal{S}_\alpha=\{K\in\mathcal{S}\colon f(K)\leq\alpha\}$ satisfying $\bar{K}\doteq K+\delta(K'-K)\in\mathcal{S}_\alpha,\forall \delta\in[0,1]$, then the following inequality holds
  $$\Vert\nabla^2f({\rm vec}(K))-\nabla^2f({\rm vec}(K'))\Vert\leq M\Vert K-K'\Vert,$$
  where $M$ is given by
  \begin{equation*}
  M=2\Vert B\Vert\Vert C\Vert\frac{\alpha^2}{\lambda_1(Q)\lambda_1(\Sigma)}(2\kappa_3+\kappa_4),
  \end{equation*}
  and $\kappa_i,i=1,2,3,4$ are defined as
  \begin{align*}
  \kappa_1&=\frac{2}{\lambda_1(Q)}\left(\Vert B\Vert\Vert C\Vert\frac{\alpha}{\lambda_1(\Sigma)}+\Vert C\Vert^2\Vert R\Vert \zeta\right),\\
  \kappa_2&=\frac{2}{\lambda_1(Q)}\left(\Vert B\Vert\Vert C\Vert\frac{\alpha}{\lambda_1(\Sigma)}+\Vert C\Vert^2\Vert R\Vert\zeta\right),\\
  \kappa_3&=\frac{2}{\lambda_1(Q)}\left(\kappa_1\Vert B\Vert\Vert C\Vert\frac{\alpha}{\lambda_1(\Sigma)}+\kappa_2\Vert B\Vert\Vert C\Vert\frac{\alpha}{\lambda_1(\Sigma)}+\Vert C\Vert^2\Vert R\Vert\right),\\
  \kappa_4&=\frac{2}{\lambda_1(Q)}\left(2\kappa_2\Vert B\Vert\Vert C\Vert\frac{\alpha}{\lambda_1(\Sigma)}+\Vert C\Vert^2\Vert R\Vert\right),
  \end{align*}
with $\zeta=\frac{2\Vert B\Vert \alpha}{\lambda_1(\Sigma)\lambda_1(R)}+\frac{\Vert A\Vert}{\Vert B\Vert}$.

\end{theorem}
\begin{proof}
The proof can be found in Section \ref{proof_of_hessian}.
\end{proof}

\begin{remark}
In addition to being a fundamental assumption of second-order optimization methods \cite{ref5}, the Lipschitz Hessian property introduced in Theorem \ref{LipschizHessian} also has a significant impact on first-order accelerated algorithms. This property has recently been demonstrated to be one of the fundamental properties that can guarantee that the Heavy-ball method holds an global accelerated convergence rate for convex problems \cite{ref6}.

\end{remark}

\section{Acceleration Optimization of SLQR Problem}

In this section, an accelerated optimization framework of SLQR problem is proposed. Specifically, we first review the structure of SLQR problem in Section \ref{section4.1}.
Then, a hybrid dynamic system is introduced in Section \ref{section4.2}, which combines the restarting method with widely known modified Heavy-ball flow. Moreover, in Section \ref{section4.3}, a discrete-time algorithm with Nesterov-optimal convergence rate is proposed that can be regarded as the discretization of the proposed hybrid dynamic system. To emphasis, the proposed restarting rule is new, which makes the convergence analysis of both the continuous-time and discrete-time algorithms differ from the one of pure modified Heavy-ball flow (see \cite{ref11}).
%In this section, we assume that $C=I$, i.e., we only focus on the SLQR setting.

\subsection{Structure of SLQR Problem}\label{section4.1}

\begin{figure}[htbp]
\centerline{\includegraphics[width=0.5\textwidth]{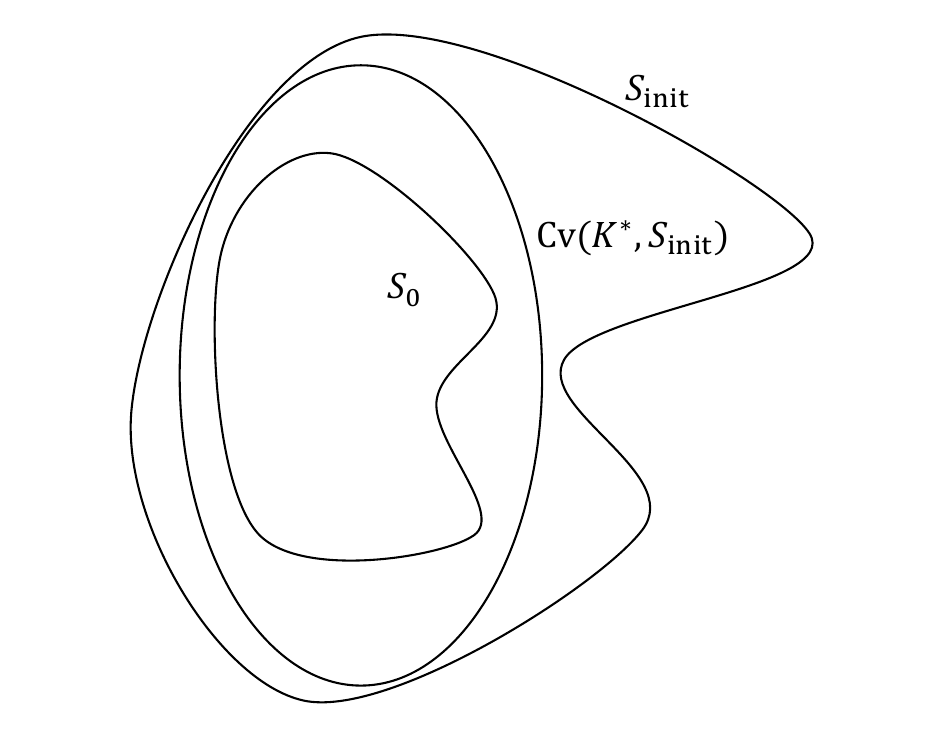}}
\caption{Structure of feasible set of SLQR problem}
\label{fig}
\end{figure}
In Figure \ref{fig}, $\mathcal{S}_{\rm init}$ is given by $\mathcal{S}_{\rm init}=\{K\in\mathcal{S}:f(K)\leq f(K_{\rm init})=\alpha_0\}$, where $K_{\rm init}$ is an initial feedback gain matrix. ${\rm Cv}(K^*,\mathcal{S}_{\rm init})$ is a convex set which is contained in $\mathcal{S}_{\rm init}$ and contains optimal feedback gain $K^*$, at the meantime, we want it to be as large as feasible. $\mathcal{S}_0$ is given by $\mathcal{S}_0=\{K\in\mathcal{S}:f(K)\leq f(K_0)=\alpha_1\}$, where $K_0\in\mathcal{S}$ ensures that the sublevel set $\mathcal{S}_0$ is contained in the convex set ${\rm Cv}(K^*,\mathcal{S}_{\rm init})$.
By Assumption \ref{assumption1}, $K_0$ is available, and the gradient descent method proposed in \cite{ref2} can serve as the basis for the selection of $K_0$.

On $\mathcal{S}_0$, $f(K)$ is $C(\alpha_0)$-smoothness with $C(\alpha_0)$ given by
$$C(\alpha_0)=\frac{2\alpha_0}{\lambda_1(Q)}(\lambda_n(R)+\sqrt{n}\Vert B\Vert\xi(\alpha_0)),$$
where
\begin{align*}
\xi(\alpha_0)&=\frac{\sqrt{n}\alpha_0}{\lambda_1(\Sigma)}\left(\frac{\alpha_0\Vert B\Vert}{\lambda_1(Q)\lambda_1(\Sigma)}+\sqrt{\left(\frac{\alpha_0\Vert B\Vert}{\lambda_1(Q)\lambda_1(\Sigma)}\right)^2+\lambda_n(R)}\right).
\end{align*}
We here use notation $C(\alpha_0)$ to underline that the estimation of smoothness parameter index depends on the initial point value $\alpha_0\doteq f(K_{\rm init})$ only. Moreover, on ${\rm Cv}(K^*,\mathcal{S}_{\rm init})$, the Hessian of $f(K)$ is $M$-Lipschitz continuous with
$$M=2\Vert B\Vert\frac{\alpha_0^2}{\lambda_1(Q)\lambda_1(\Sigma)}(2\kappa_3+\kappa_4),$$
where $\kappa_3,\kappa_4$ are defined in Theorem \ref{LipschizHessian}. We denote
\begin{align*}
	\mu&=\min_{E\in\mathbb{R}^{m\times n},\Vert E\Vert_F=1}\nabla^2 f(K)[E,E]\bigg|_{K=K^*},\\
	L&=\max_{E\in\mathbb{R}^{m\times n},\Vert E\Vert_F=1}\nabla^2 f(K)[E,E]\bigg|_{K=K^*},
\end{align*}
and based on the smoothness of $f(K)$ we can show that
$$L\leq C(f(K^*)).$$
$f(K)$ satisfies the PL condition (with parameter $\mu(f(K_\chi))$) in any sublevel set $\mathcal{S}'=\{K\in\mathcal{S},f(K)\leq f(K_\chi)\}$ with $K_{\chi}\in\mathcal{S}$, i.e.,
$$\frac{1}{2}\Vert \nabla f(K)\Vert_F^2\geq \mu(f(K_\chi))(f(K)-f(K^*)),$$
where
$$\mu(f(K_\chi))=\frac{\lambda_1(R)\lambda_1(\Sigma)^2\lambda_1(Q)}{8f(K^*)\left(\Vert A\Vert+\frac{\Vert B\Vert^2f(K_\chi)}{\lambda_1(\Sigma)\lambda_1(R)}\right)^2}.$$
According to Theorem 2 of \cite{ref8}, it is shown that the PL condition implies that the function has quadratic growth, i.e.,
\begin{equation*}
	f(K)-f(K^*)\geq \frac{\mu(f(K_\chi))}{2}\Vert K-K^*\Vert_F^2.
\end{equation*}
Thus, we can easily demonstrate that
$$\mu\geq\mu(f(K^*)).$$
In this section and  without causing confusion, we denote
\begin{equation*}
	\mu=\mu(f(K^*)),~ L=C(f(K^*)),~ \kappa=\frac{L}{\mu}.
\end{equation*}

\subsection{Continuous-time Analysis}\label{section4.2}
\begin{definition}[Hybrid Dynamic System \cite{ref9}]\label{Defi-2}
  A time-invariant hybrid dynamic system ($\mathcal{H}$) comprises a ODE and a jump rule introduced as
  \begin{equation}\label{H0}
    \left\{
    \begin{aligned}
    \dot x&=F(x),~~ x\in\mathcal C,\\
    x^+&=G(x),~~\text{otherwise},\\
    \end{aligned}
    \right.
  \end{equation}
  where $x^+$ is the state after a jump; the function $F:\mathbb{R}^n\longmapsto\mathbb{R}^n$ is the flow map, and the set $\mathcal{C}\subset\mathbb{R}^n$ is the flow set, and the function $G:\partial\mathcal{C}\longmapsto {\rm int}(\mathcal{C})$ represents the jump map.
\end{definition}

In this section, we denote $z(t)=(K(t)^{\rm T},\dot{K}(t)^{\rm T})^{\rm T}\in\mathbb{R}^{2m\times n}$. The flow map $F:\mathbb{R}^{2m\times n}\longmapsto\mathbb{R}^{2m\times n}$  is given by
\begin{equation}\label{H1}
	F(z(t))=\begin{pmatrix}
		\dot{K}(t)\\
		-\nabla f(K(t))+f_d(K(t),\dot{K}(t))
	\end{pmatrix},
\end{equation}\label{H2}
with
$$f_d(q,p)=-2dp-(\nabla f(q+\beta p)-\nabla f(q)),$$
and the flow set $\mathcal{C}\subseteq \mathbb{R}^{2m\times n}$ is given by
\begin{equation}
	\mathcal{C}=\left\{z(t)\in\mathbb{R}^{2m\times n}\colon f(K(t))<f(K_0)\bigvee\frac{{\rm d}}{{\rm d}t}f(K(t))<0\right\},
\end{equation}
where $K_0$ is defined in Section \ref{section4.1}. The jump map $G:\mathbb{R}^{2m\times n}\longmapsto \mathbb{R}^{2m\times n}$ parameterized by a constant $\eta$ is given by
\begin{equation}\label{H3}
	G(z)=
	\begin{pmatrix}
		K(t)\\
		-\eta\nabla f(K(t))
	\end{pmatrix}.
\end{equation}
Here, we have so far succeeded in creating a complete hybrid dynamic system and we will then show the hybrid dynamic system with parameters (\ref{H1})-(\ref{H3}) is well-defined.

\begin{remark}\label{history}
Continuous-time dynamic system $\dot{x}=F(x)$ with $F$ given by (\ref{H1}), known as the modified Heavy-ball system, is first introduced in \cite{ref11} and has been widely studied by many scholars \cite{Cortes-2022, z7, ref10}. In this paper, we mostly consider the case $\beta=0$ (in (\ref{H1})), and the dynamic system becomes
\begin{equation*}
    \ddot{K}(t)+2d\dot{K}(t)+\nabla f(K(t))=0,
\end{equation*}
which is widely known as the Heavy-ball dynamic system \cite{Alvarez-SICON2000}. Recently, there has been a resurgence of scholarly interest in Heavy-ball dynamical system; see \cite{z1,z3,z4,z6}. For SLQR problem, the feasible set is nonconvex, and the $L$-smoothness and PL condition only hold on some compact sublevel set, making the methodology of convergence analysis differ from those of \cite{z1,z3,z4,z6}.  Hence, the hybrid dynamic system (\ref{H0}) with parameters (\ref{H1})-(\ref{H3}) does not exhibit parallel generalizations of abovementioned papers.
\end{remark}

\begin{lemma}
The hybrid dynamic system (\ref{H0}) with parameters (\ref{H1})-(\ref{H3}) is well-defined if $K(0)\in\mathcal{S}_0$, where $\mathcal{S}_0$ is defined in Figure \ref{fig}.
\end{lemma}
\begin{proof}
  We only need to prove that the state must enter the flow set $\mathcal{C}$ after a jump. Based on the definition of $\mathcal{C}$ and the continuity of flow map $F(z)$, the jump map is executed when $f(K(t))=f(K(0))$ and $\frac{\rm d}{{\rm d}t}f(K(t))\geq 0$. Then,
  	$$\frac{{\rm d}}{{\rm d}t}f(K(t))=\langle\nabla f(K(t)),\dot{K}(t)\rangle=-\eta\Vert \nabla f(K(t))\Vert_F^2<0,$$
  i.e., the objective function $f(K(t))$ is locally strictly monotone decreasing and the flow map will be executed. Hence, the hybrid dynamic system is well-defined.
\end{proof}

\begin{assumption}
Assume that the jump map will be executed no more than $S$ times while the hybrid dynamic system (\ref{H0}) with  parameters (\ref{H1})-(\ref{H3}) is in operation.
\end{assumption}

Let us define the total energy function
$H:\mathbb{R}^{m\times n}\times \mathbb{R}^{m\times n}\to\mathbb{R}$:
\begin{equation*}
  H(q,p)=\frac{1}{2}\Vert p\Vert^2_F+f(q)-f(K^*).
\end{equation*}
\begin{definition}
Assuming that $F(z)$ is continuously differential, we call the hybrid dynamic system (\ref{H0}) a momentum-based optimization algorithm with restarting rule $G(z)$ for $f(\cdot)$, if the local minimizer denoted by $z^*$ is an asymptotically stable equilibrium in the sense of Lyapunov, i.e., $\varphi_t(z^*)=z^*$, for all $t\in\mathbb{R}_+$, $\varphi_t$ is continuous at $z^*$, uniformly in $t$, and $\lim_{t\to\infty}\varphi_t(z_0)=z^*$ for any $z_0$ in a neighborhood of $z^*$, where $\varphi_t$ denotes the solution trajectory of $\dot{z}=F(z)$. In addition, the jump map should be executed in finite times. Moreover, the region of attraction of the equilibrium $z^*$ (of the hybrid dynamic system (\ref{H0})) is defined as the set
\begin{equation*}
  \mathcal{R}=\{z_0\in\mathbb{R}^{2m\times n}\colon \lim_{t\to\infty}\varphi_t(z_0)=z^*\}.
\end{equation*}
\end{definition}
\begin{theorem}\label{theorem4}
  Provided that $d>0$ and $\beta=0$, the equilibrium $z^*$ is  asymptotically stable in the sense of Lyapunov with
  $$z^*=\begin{pmatrix}
  K^*\\
  0
  \end{pmatrix},$$
and its region of attraction contains the set $\mathcal{A}_f=\mathcal{S}_1\times\mathbb{R}^{m\times n}$.
\end{theorem}
\begin{proof}
First, we will show that the SLQR cost function $f(K)$ has a unique stationary point $K^*$ that is globally optimal. For any stationary point $K_*$, it is straightforward to show that
\begin{align*}
 \frac{1}{2}\nabla^2 f(K_*)[E,E]&=\langle REY,E\rangle-2\langle B^{\rm T}X'Y,E\rangle\\
 &=\langle REY,E\rangle\\
 &={\rm Tr}\left((R^{\frac{1}{2}}E)Y(R^{\frac{1}{2}}E)^{\rm T}\right)>0,
\end{align*}
where the last inequality is implied by $R,Y\succ 0$. Hence, $f(K)$ is strongly convex in the neighborhood of $K_*$. Combining the $\mathcal{C}^1$-continuity and PL condition of $f(K)$, we have the uniqueness and globally optimal guarantee of the stationary point.

Noticing that when the flow map is executed, we have
\begin{align*}
		&\quad\frac{\rm d}{{\rm d}t}H(K(t),\dot{K}(t))\\
		&=\langle \dot{K}(t),\ddot{K}(t)\rangle+\langle \nabla f(K(t)),\dot{K}(t)\rangle\\
		&=\langle \dot{K}(t),-2d\dot{K}(t)-\nabla f(K(t))\rangle+\langle \nabla f(K(t)),\dot{K}(t)\rangle\\
		&=-2d\Vert\dot{K}(t)\Vert_F^2\leq0.
	\end{align*}
 Hence, for any initial point $z(0)\in\mathcal{A}_f$, the trajectory $z(t)$ is confined to a connected component of $H^ {-1}([0,H_0])$ with $H_0=H(z(0))$. By the definition of $H$ we infer that the stationary points of $H$ are all of the form $(K_*,0)$, where $K_*$ corresponds to a stationary point of $f(K)$. Obviously, there is unique stationary point $(K^*,0)$ and $H$ is strongly convex in its neighborhood. The abovementioned fact implies that $H$ has no such stationary point which has a strong Morse index $m_*$ and a weak Morse index $m^*$ related by $m_*\leq 1\leq m^*$. Combing the coercivity and smoothness of $H$, we can conclude that for any  $z(0)\in\mathcal{A}_f$, $H^{-1}([0,H_0])$ is path-connected by mountain pass theorem. Also, we can observe that $H^{-1}([0,H_0])$ is compact. This implies that $z^*$ is stable.

 We then show the region of attraction of $z^*$ contains the set $\mathcal{A}_f$. $H$ is a strictly decreasing function except the case when $\dot{K}(t)=0$. Hence according to La Salle's theorem, $K(t)$ necessarily converges to a stationary  point of $f$ and $\dot{K}(t)$ converges to 0. $K^*$ is the only stationary point in ${\rm Proj}_1\mathcal{A}_f$, and this implies the asymptotically stability of $z^*$.
\end{proof}

\begin{remark}\label{remark3}
By taking $z(0)=(K(0),0)$, the solution trajectory of $\dot{x}=F(x)$ with (\ref{H1}) is confined to
$$H^{-1}([0,H_0])=\{z:\frac{1}{2}\Vert \dot{K}(t)\Vert^2_F+f(K(t))\leq f(K(0))\},$$
and the jump map will never be executed, i.e., the hybrid dynamic system with parameters (\ref{H1})-(\ref{H3}) degenerates into a  continuous-time system.
However, we aim to broaden the range of options for the initial point $z(0)=(K(0),p(0))$, with the condition that $p(0)$ may be not $0$. In this case, $H^{-1}([0,H_0])$ becomes
$$H^{-1}([0,H_0])=\{z:\frac{1}{2}\Vert \dot{K}(t)\Vert^2_F+f(K(t))\leq f(K(0))+\frac{1}{2}\Vert p(0)\Vert_F^2\}.$$
Then, it is possible that $f(K(t))<f(K(0))$ for some $t\in\mathbb{R}_+$, i.e., the jump map may be executed. Therefore, we make the assumption that the jump map will be executed no more than $S$ times.

In addition, we want to provide an elucidation from the perspective of stability theory. By Theorem \ref{theorem4}, $z^*=(K^*,0)$ is asymptotically stable in the sense of Lyapunov: for every $\epsilon>0$, there exists a $\delta(\epsilon)>0$ such that, if $\Vert z(0)-z^*\Vert\leq \delta(\epsilon)$, then for every $t\in\mathbb{R}_+$ we have $\Vert z(t)-z^*\Vert\leq \epsilon$ and $z(t)\to z^*$. Hence, if we choose $\epsilon$ small enough such that $B_{\delta(\epsilon)}(\epsilon)\subseteq \mathcal{A}_f$, then the jump map will never be executed. The sole remaining task is to select a ``good'' initial point $\tilde{z}_0$ such that $\Vert \tilde{z}_0-z^*\Vert\leq\delta(\epsilon)$. Moreover, gradient descent (introduced in \cite{ref2}) can be used to select this ``good'' initial point $\tilde{z}_0$ (see Theorem 4.2 of \cite{ref2} for details). Yet again, we hope to expand the selection range of initial point, i.e., $\epsilon$ does not need to be so small (equivalently, initial point $\tilde{z}_0$ does not need to be so ``good'') such that $B_{\delta}(\epsilon)\subseteq \mathcal{A}_f$. When $K(t)$ is about to escaping from $\mathcal{S}_1$, the jump map will be executed, which can confine $K(t)$ for all $t\in\mathbb{R}_+$ to $\mathcal{S}_1$. Hence, we make the assumption that the jump map will be executed no more than $S$ times.

\end{remark}

\begin{theorem}
  Provided that $d\sim 1/\sqrt{\kappa},0<d$ and $\beta=0$, there exists a sufficiently small $\eta>0$, such that the hybrid dynamic system with parameters (\ref{H1})-(\ref{H3}) is accelerated; that is, there exists constants $C_a,c_a>0$, such that $\forall t\in\mathbb{R}_+$ and $\forall z_0\in\mathcal{A}_f$ the following bound holds:
  \begin{equation}\label{th1}
    \Vert\varphi_t(z_0)-z^* \Vert_F\leq C_a\Vert z_0-z^*\Vert_F\exp(-c_at/\sqrt{\kappa});
  \end{equation}
  here, $\varphi_t(z_0)$ denotes the solution trajectory of the hybrid dynamic system with initial point $z_0$.
\end{theorem}

\begin{proof}
Let $t_k$ be the time instant of the $k$-th execution of the jump map, and $z_0^k$ be the initial point of the $k$-th flow map. Moreover, $\varphi_t^k(z_0^k)$ is the position  at time instant $t$ of solution trajectory of the $k$-th flow map with initial point $z_0^k$. We have
\begin{align*}
&\quad\Vert\varphi_t^k(z_0^k)-z^*\Vert_F\\
&\leq\tilde{C}_a\Vert z_0^k-z^*\Vert_F\exp\left(-\frac{c_a}{\sqrt{\kappa}}(t-t_k)\right)\\
&\leq\tilde{C}_a\Vert\varphi_t^{k-1}(z_0^{k-1})-z^*\Vert_F\exp\left(-\frac{c_a}{\sqrt{\kappa}}(t-t_k)\right)\\
&\leq\tilde{C}_a^2\Vert z_0^{k-1}-z^*\Vert_F\exp\left(-\frac{c_a}{\sqrt{\kappa}}(t-t_{k-1})\right)\\
&~~\vdots\\
&\leq\tilde{C}_a^{k+1}\Vert z_0-z^*\Vert_F\exp\left(-\frac{c_a}{\sqrt{\kappa}}t\right),
\end{align*}
where the first inequality follows from Proposition 8 of \cite{ref10}, and the second inequality holds as $\eta$ is small enough. According to the assumption that the jump map will be executed no more that $S$ times, we can demonstrate (\ref{th1}) by denoting $C_a=\tilde{C}_a^{S+1}$.
\end{proof}

\subsection{Discrete-time Analysis}\label{section4.3}

To obtain the discrete-time algorithm, we discretize the flow map (\ref{H1}) by semi-implicit Euler method
\begin{equation}\label{d1}
\begin{aligned}
	 K_{k+1}&=K_k+Tp_{k+1},\\
   p_{k+1}&=p_k-T\nabla f(K_k)+Tf_d(K_k,p_k)\\
	 &=(1-2dT)p_k-T\nabla f(K_k+\beta p_k),
\end{aligned}
\end{equation}
which can be divided into two parts: an energy dissipation step and a symplectic Euler step \cite{ref11}. Due to the special strutrue of SLQR problem mentioned in Section \ref{section4.1}, we hope that the discrete-time trajectory of $K_k$ will be confined to the sublevel set $\mathcal{S}_0$. We consider the special case  $\beta=0$ and propose the following restarting rule
\begin{equation}\label{restart1}
\begin{aligned}
\text{if } &f(K_{k+1})>\alpha_1,\text{ then restart with }K_0=K_k,p_0=-\eta\nabla f(K_k).
\end{aligned}
\end{equation}

\begin{lemma}\label{lemma4}
Provided that $\beta=0$ and $\eta$ is small enough, the discrete-time algorithm (\ref{d1}) with restarting rule (\ref{restart1}) is well defined when the step size $T$ satisfies
$$0<T\leq\frac{-\eta+\sqrt{\eta^2+\frac{8(1-2d\eta)}{C(\alpha_0)}}}{2
	(1-2d\eta)}.$$
\end{lemma}
\begin{proof}
After restarting, the first step of iteration becomes
$$K_{k+1}=K_k-T(T+(1-2dT)\eta)\nabla f(K_k),$$
which is equivalent to vanilla gradient descent. Let $K^{t}=K_{k}-t\nabla f(K_k),\phi(t)=f(K^{t})$, then
$$\phi'(0)=-\Vert \nabla f(K_k)\Vert_F^2<0.$$
Thus $\phi(t)<\phi(0)=f(K_k)$ and $K^t\in\mathcal{S}_0$ for sufficiently small $t>0$. From the compactness of $\mathcal{S}_0$, denote
$$T=\max\{t\colon K^\tau\in\mathcal{S}_0,0\leq\tau\leq t\}.$$
Then, for $0\leq t\leq T$, $f(K^t)$ is $C(\alpha_0)$-smooth. Hence, $\phi(t)$ is $L'$-smooth with $L'=L\Vert\nabla f(K_k)\Vert_F^2,$ and we have
$$\vert\phi(T)-\phi(0)-\phi'(0)T\vert\leq\frac{L'T^2}{2}.$$
Since $\phi(T)=\phi(0)=f(K_k)$, we conclude that $T\geq\frac{2\vert\phi'(0)\vert}{L'}=\frac{2}{C(\alpha_0)}.$ This means that $K^t\in\mathcal{S}_0$ for all $t\in\left[0,\frac{2}{C(\alpha_0)}\right] $. Hence, we can ensure $f(K_{k+1})\in\mathcal{S}_0$ when
$$T(T+ (1-2dT)\eta)\leq\frac{2}{C(\alpha_0)}.$$
By solving the inequality above, we have
$$0<T\leq\frac{-\eta+\sqrt{\eta^2+\frac{8(1-2d\eta)}{C(\alpha_0)}}}{2
	(1-2d\eta)}.$$
This completes the proof.
\end{proof}

Similar to the continuous-time case (Theorem \ref{theorem4}), the following lemma provides the domain of attraction for the discrete-time system (\ref{d1}).

\begin{lemma}[\cite{ref10}]\label{lemma5}
Assume that the trajectory of discrete-time system (\ref{d1})  satisfies that $\{K_k\}\subseteq {\rm Cv}(K^*,\mathcal{S}_{\rm init})$, and let $f_d(K_k,p_k)$ be such that $-d_2\Vert p_k\Vert^2_F\leq \langle p_k,f_d(K_k,p_k)\rangle\leq -d_1\Vert p_k\Vert^2_F$ with $0<d_1\leq d_2$. Then, there exists a maximum time step $T_{\rm max}>0$, such that $z^*$ is an asymptotically stable equilibrium of dynamics (\ref{d1})  for all $T\leq T_{\rm max}$, with the domain of attraction at least containing $\mathcal{A}_f\cap A$, where $A$ is any compact set. Moreover, the maximum time step $T_{\rm max}$ depends on an upper bound on $d_2,d_1/d_2,L_{\rm Ham}$, where $L_{\rm Ham}$ represents the Lipschitz constant of $\nabla H$.
\end{lemma}

The following theorem provides the convergence analysis of discrete-time system (\ref{d1}), showing that iteration algorithm (\ref{d1}) exhibits Nesterov-optimal convergence rate.

\begin{theorem}\label{theorem6}
  Letting $d\sim 1/\sqrt{\kappa},d>0$ and $\beta=0$ and choosing $T$ sufficiently small such that
  $$0<T\leq\min\left\{T_{\rm max},\frac{-\eta+\sqrt{\eta^2+\frac{8(1-2d\eta)}{C(\alpha_0)}}}{2
(1-2d\eta)}\right\},$$
the discrete-time system (\ref{d1})  with the restarting rule (\ref{restart1}) is accelerated; that is,
\begin{equation*}
\Vert(K_k,p_k)-z^*\Vert_F=\mathcal{O} \left(\exp\left({-\frac{c}{\sqrt{\kappa}}}k\right)\right),
\end{equation*}
where $c>0$ is a constant.
\end{theorem}
\begin{proof}
  By the restarting rule introduced by (\ref{restart1}), we can ensure that the trajectory of $K_k$ is confined to the sublevel set $\mathcal{S}_0\subseteq{\rm Cv}(K^*,\mathcal{S}_{\rm init}),$ and have
  $$\langle p_k,f_d(K_k,p_k)\rangle=-2d\Vert p_k\Vert^2_F.$$
  Hence, the assumption of Lemma \ref{lemma5} naturally holds, and we have thus far determined $z^*$'s domain of attraction.
Then, based on Lemma \ref{lemma4}, this theorem can be deduced directly from the Proposition 11 of \cite{ref10}.
\end{proof}
\begin{remark}
To emphasize, \cite{ref2} suggests that the SLQR problem can be solved by vanilla gradient descent and draws the conclusion that the iteration sequence fulfills the following convergence rate
\begin{equation}\label{Polyak}
\Vert K_k-K^*\Vert\leq \mathcal{O}\left(\exp\left(\frac{\tilde{c}}{\tilde{\kappa}}k\right)\right),
\end{equation}
where $\tilde{c}$ is a constant and $\tilde{\kappa}$ is given by
$\tilde{\kappa}=\frac{C(\alpha_0)}{\mu(\alpha_0)}.$
We may consider $\kappa$ as a strictly increasing function of the initial value, i.e.,
$\kappa(\alpha)=\frac{C(\alpha)}{\mu(\alpha)}.$
Hence, (\ref{Polyak}) is equivalent to
$$\Vert K_k-K^*\Vert\leq \mathcal{O}\left(\exp\left(\frac{\tilde{c}}{\kappa(\alpha_0)}k\right)\right),$$
and the convergence rate given in Theorem \ref{theorem6} is equivalent to
$$\Vert K_k-K^*\Vert\leq \mathcal{O}\left(\exp\left(\frac{c}{\sqrt{\kappa(f(K^*))}}k\right)\right).$$
As a result, our method can not only achieve the Nesterov-optimal convergence rate, but also ensure that the choice of initial point makes no difference on the estimation of  convergence order.
\end{remark}

\section{Acceleration Optimization of OLQR Problem}

In this section, we will discuss the acceleration optimization of the OLQR problem, and due to problem's nonconvex structure, we only pursue the stationary point instead of the global minimizer of objective function. By using vanilla gradient descent, it has been demonstrated that under suitable conditions the following convergence result holds \cite{ref2}
%{sec5-1}
\begin{equation}\label{sec5-1}
\min_{0\leq j\leq k}\Vert \nabla f(K_j)\Vert_F \leq \sqrt{\frac{f(K_0)}{c_1 k}},
\end{equation}
where $c_1$ is a constant. Moreover, \cite{ref12} examines the general nonconvex optimization problems from the perspective of Hamiltonian by generalized momentum-based methods, and shows the convergence result by some variants of NAG
\begin{equation}\label{sec5-2}
\min_{0\leq j\leq k}\Vert \nabla f(K_j)\Vert_F = \mathcal{O}\left(\sqrt{\frac{L(f(x_0)-f^*)}{k}}\right),
\end{equation}
where $L$ is a smoothness parameter.
Interestingly, the result of (\ref{sec5-2}) is not better than that of (\ref{sec5-1}),
and it is guessed that utilizing NAG directly cannot obtain an accelerated nonconvex optimization algorithm.
Therefore, the acceleration optimization of OLQR problem with a completely nonconvex structure is not a natural extension of that of SLQR problem.

In this section, motivated by \cite{ref13}, we introduce a new Hessian-free algorithm framework which is a two-procedure method  consisting of semiconvex function optimization and negative curvature exploitation. Different from \cite{ref13}, a new restarting rule is proposed based on the two-procedure method, requiring additional effort for convergence analysis. Specifically, Section \ref{section5.1} and Section \ref{section5.2} introduce the basic results on semiconvex function optimization and negative curvature exploitation with restarting rule, respectively. Further, in Section \ref{section5.3}, we review the structure of OLQR problem and propose a two-procedure algorithm framework with restarting rule for $\epsilon$-stationary point with accelerated convergence rate.

\subsection{Semiconvex Function Optimization}\label{section5.1}

\begin{definition}[generalized strong convexity and semi-convexity]
  A function $\phi\colon\mathbb{R}^d\longmapsto\mathbb{R}$ is $\sigma_1$-generalized-strongly convex if $\frac{\sigma_1}{2}\Vert y-x \Vert^2\leq \phi(y)-\phi(x)-\langle \nabla \phi(x),y-x\rangle$ for some $\sigma_1\in\mathbb{R}$; equivalently, $\nabla^2\phi(x)\succeq \sigma_1 I$ for all $x$. For $\gamma=\max\{-\sigma_1,0\}$, we call such function $\gamma$-semiconvex.
\end{definition}
\begin{definition}[optimality gap]
  A function $\phi\colon\mathbb{R}^d\longmapsto\mathbb{R}$ has optimality gap $\Delta_\phi$ at point $x$ if $\Delta_\phi\geq \phi(x)-\inf_{y\in\mathbb{R}^d}\phi(y)$.
\end{definition}

Before introducing Semiconvex-NAG, which is an accelerated stationary point-seeking method for semiconvex functions, we present NAG with a restarting rule for strongly convex functions.

\begin{breakablealgorithm}%[H]
    \caption{Nesterov-Accelerated-Gradient (NAG) }%算法标题
    \begin{algorithmic}[1]%一行一个标行号
    \State \textbf{Given} : Object function $\phi$, initial point $y_1$, accuracy $\epsilon$, smoothness parameter $L_1$, convexity parameter $\sigma_1$
    \State \textbf{Result} : $K_j$
    \State Set $\kappa=\frac{L_1}{\sigma_1}$, and $K_1=y_1$
    \For{$j=1,2,...$}
        \If{$\Vert\nabla \phi(K_j)\Vert<\epsilon$}
           \Return {$K_j$}
        \EndIf
        \State Let
        \begin{align*}
	   y_{j+1}&=K_j-\frac{1}{L_1}\nabla \phi(K_j)\\ K_{j+1}&=\left(1+\frac{\sqrt{\kappa}-1}{\sqrt{\kappa}+1}\right)y_{j+1}-\frac{\sqrt{\kappa}-1}{\sqrt{\kappa}+1}y_j
        \end{align*}
    \EndFor
    \end{algorithmic}
    \label{alg1}
\end{breakablealgorithm}

In this section, the algorithm's output is denoted by the name of algorithm (or its abbreviation) along with its inputs. For instance, Nesterov-Accelerated-Gradient$(\phi,y_1,\epsilon,L_1,\sigma_1)$ (or abbreviation NAG$(\phi,y_1,\epsilon,L_1,\sigma_1)$) denotes the result of algorithm NAG with inputs $(\phi,y_1,\epsilon,L_1,\sigma_1)$.

We introduce the following restarting rule:
\begin{equation}\label{restart}
  \text{if}~\phi(K_{j+1})\geq \phi(K_1)=\eta,~\text{restart NAG}(\phi,K_j,\epsilon ,L_1,\sigma_1).
\end{equation}
\begin{lemma}\label{lemma6}
Assume that the restarting rule is executed no more that $S$ times and let $\phi\colon\mathbb{R}^{m\times n}\to \mathbb{R}$ be $(\sigma_1>0)$-strongly convex and $L_1$-smooth. Let $\epsilon>0$ and $K_N$ be the $N$-th iteration of NAG$(\phi,K_1,\epsilon,L_1,\sigma_1)$ with restarting rule (\ref{restart}). Then $\Vert \nabla \phi(K_N)\Vert_F\leq\epsilon$ for every
$$N\geq S+1+\sqrt{\kappa}\log\left(\frac{2^{S+2}\kappa^{S+1}L_1\Delta_\phi}{\epsilon^2}\right),$$
where $\kappa=\frac{L_1}{\sigma_1}$ is the condition number of $\phi$ and $\Delta_\phi$ is optimality gap of $\phi$ at $K_1$.
\end{lemma}
\begin{proof}
First, we will show that the restarting rule is well defined, i.e., if $K\in\hat{\mathcal{S}}_0$, then
$$\hat{K}={\rm NAG}_1(\phi,K,\epsilon ,L_1,\sigma_1)\in\hat{\mathcal{S}}_0,$$
where $\hat{\mathcal{S}}_0=\{K\in\mathbb{R}^{m\times n}\colon \phi(K)\leq\eta\}$, and ${\rm NAG}_1$ denotes the result of the one-step iteration of Algorithm 1. Obviously, the following equality holds
$$\hat{K}=K-\frac{1+\frac{\sqrt{\kappa}-1}{\sqrt{\kappa}+1}}{L_1}\nabla \phi(K).$$
Again, by denoting $\varphi(t)=\phi(K_t)=\phi(K-t\nabla \phi(K))$, we can demonstrate that for all $0\leq t\leq \frac{2}{L_1}$ we have $K_t\in\hat{\mathcal{S}}_0$. Hence, $\hat{K}\in\hat{\mathcal{S}}_0$ and the restarting rule is well-defined.

Nesterov \cite{ref5} shows that for any $j>1$,
$$\phi(K_j)-\phi(K^*)\leq L_1\left(1-\sqrt{\frac{\sigma_1}{L_1}}\right)^{j-1}\Vert K_1-K^*\Vert^2_F.$$
Let $K_j^{(i)}$ represent the result of the $j$-th algorithm iteration, after the $i$-th execution of the restarting rule but prior to the $(i+1)$-th execution of the restarting rule. Let function $j(i)$ be the number of iterations between the $i$-th execution of the restarting rule and the $(i+1)$-th execution of the restarting rule (or the occurrence of the algorithm's termination condition). Since $\phi$ is strongly convex, the following quadratic growth condition holds \cite{ref5}
$$\phi(K)-\phi(K^*)\geq \frac{\sigma_1}{2}\Vert K-K^*\Vert^2_F.$$
Thus, we have
\begin{align*}
		&\quad \phi(K_j^{(i)})-\phi(K^*)\\
&\leq L_1\left(1-\sqrt{\frac{\sigma_1}{L_1}}\right)^{j-1}\Vert K_1^{(i)}-K^*\Vert^2_F\\
		&= L_1\left(1-\sqrt{\frac{\sigma_1}{L_1}}\right)^{j-1}\Vert K_{j(i-1)}^{(i-1)}-K^*\Vert^2_F\\
		&\leq L_1\frac{2}{\sigma_1}\left(1-\sqrt{\frac{\sigma_1}{L_1}}\right)^{j-1}\left(\phi(K_{j(i-1)}^{(i-1)})-\phi(K^*)\right)\\		
		&\leq \left(\frac{2L_1}{\sigma_1}\right)^i\left(1-\sqrt{\frac{\sigma_1}{L_1}}\right)^{\sum_{p=1}^{i-1}j(p)+j-i}\left(\phi(K_{j(0)}^{(0)})-\phi(K^*)\right)\\
		&\leq \left(\frac{2L_1}{\sigma_1}\right)^{i+1}\left(1-\sqrt{\frac{\sigma_1}{L_1}}\right)^{N-i-1}\left(\phi(K_1)-\phi(K^*)\right)\\
		&\leq 2^{i+1}\kappa^{i+1}\exp(-(N-i-1)\kappa^{-\frac{1}{2}})\Delta_\phi\\
		&\leq 2^{S+1}\kappa^{S+1}\exp(-(N-S-1)\kappa^{-\frac{1}{2}})\Delta_\phi.	
	\end{align*}
Taking any $$N\geq S+1+\sqrt{\kappa}\log\left(\frac{2^{S+2}\kappa^{S+1}L_1\Delta_\phi}{\epsilon^2}\right)$$ yields
$$\phi(K_N)-\phi(K^*)\leq\frac{\epsilon^2}{2L_1}.$$
Noticing that $\Vert \nabla \phi(K)\Vert_F^2\leq 2L_1(\phi(K)-\phi(K^*)) $ \cite{ref5}, we obtain the result.
\end{proof}

Next, we will examine how to find the stationary point of a $\gamma$-semiconvex function $\psi$.
Semiconvex functions have been extensively investigated in the field of nonconvex optimization \cite{Mifflin-1977,Ngai-2023,ref13}, and the fundamental idea is to add a regularizing term $\varphi=\gamma\Vert K-K_j\Vert_F^2$ such that $\psi+\varphi$ is $\gamma$-strongly convex. Semiconvex-NAG is first proposed in \cite{ref13}, and different from \cite{ref13}, the sub-algorithm NAG of Algorithm \ref{alg2} below is executed with restarting rule, requiring additional discussion of complexity analysis (see Theorem \ref{thm7}).

\begin{breakablealgorithm}%[H]
    \caption{Semiconvex-NAG}%算法标题
    \begin{algorithmic}[1]%一行一个标行号
    \State \textbf{Given} : Object function $\psi$, initial point $K_1$, accuracy $\epsilon$, smoothness parameter $L_1$, semi-convexity parameter $\gamma$
    \State \textbf{Result} : $K_j$
    \State Set $\kappa=\frac{L_1}{\sigma_1}$, and $K_1=y_1$
    \For{$j=1,2,...$}
        \If{$\Vert\nabla \psi(K_j)\Vert<\epsilon$}
           \Return {$K_j$}
        \EndIf
        \State Let $g_j(K)=\psi(K)+\gamma\Vert K-K_j\Vert^2_F$
        \State $\epsilon'=\epsilon\sqrt{\gamma/(50(L_1+2\gamma))}$
        \State $K_{j+1}\leftarrow {\rm NAG}(g_j,K_j,\epsilon',L_1+2\gamma,\gamma)$(with restarting rule (\ref{restart}))
    \EndFor
    \end{algorithmic}
    \label{alg2}
\end{breakablealgorithm}

\begin{assumption}\label{assump2}
  The following operations take $T_{\rm grad}$ time:
  \begin{enumerate}
    \item The evaluation $\nabla h(K)$ for a point $K\in\mathbb{R}^{m\times n}$.
    \item The evaluation of $\nabla^2 h(K)v$ for some direction $v\in\mathbb{R}^{mn}$ and point $K\in\mathbb{R}^{mn}$.
    \item Any arithmetic operation (affine combinations and inner products) of two vectors of dimension at most $m\times n$.
  \end{enumerate}
\end{assumption}
\begin{remark}
  There are a few aspects that require justification in Assumption \ref{assump2}. First, to avoid tensors, we consider the  domain of function $h$ to be $\mathbb{R}^{mn}$ rather than $\mathbb{R}^{m\times n}$. Second, it appears inevitable that we must compute the Hessian of function $h$ to derive the value $\nabla^2 h(K)v$. However, by the definition of  directional derivative of $\nabla h(K)$ in direction $v\in\mathbb{R}^{m\times n}$, we have
  $$\lim_{h\to 0}\frac{\nabla h(K+hv)-\nabla h(K)}{h}=\nabla^2 h(K)v,$$
  and this leads a natural approximation of $\nabla^2 h(K)v$:
  $$p=\frac{\nabla h(K+hv)-\nabla h(K)}{h}$$
for sufficient small $h$ that is Hessian-free.
  In addition, \cite{ref13} provides an error analysis for the abovementioned approximation method and states that the method permits adequately precise calculations when $h$ is sufficiently small.
\end{remark}
\begin{theorem}\label{thm7}
  Assume that the execution of restarting rule (\ref{restart}) is no more than $S$ times for each call of algorithm NAG. Let $\sigma_1>0$ and $\psi:\mathbb{R}^{m\times n}\to\mathbb{R}$ be $\sigma_1$-semiconvex and $L_1$-smooth, and let $\gamma\in[\sigma_1,L_1]$. Then, Semiconvex-NAG$(\psi,K_1,\epsilon,\gamma,L_1)$ returns a point $K$ such that $\Vert \nabla \psi(K)\Vert_F\leq\epsilon$ and
  \begin{equation}\label{e1}
  \psi(K_1)-\psi(K)\geq \min\left\{\gamma\Vert K-K_1\Vert_F^2,\frac{\epsilon}{\sqrt{10}}\Vert K-K_1\Vert_F\right\}
  \end{equation}
  in time
  \begin{equation}\label{time1}
  \begin{aligned}
    \mathcal{O}&\left({T_{\rm grad}}\left(\sqrt{\frac{L_1}{\gamma}}+\frac{\sqrt{\gamma L_1}}{\epsilon^2}(\psi(K_1)-\psi(K))\right)\log\left(\frac{L_1^{S+3}\Delta_\psi}{\gamma^{S+2}\epsilon^2}\right)\right).
    \end{aligned}
  \end{equation}
\end{theorem}
\begin{proof}
By denoting the time at which routine terminates as $j_*$, \cite{ref13} shows
$$j_*\leq 1+\frac{5\gamma}{\epsilon^2}[\psi(K_1)-\psi(K_{j_*})]\leq 1+\frac{5\gamma}{\epsilon^2}\Delta_\psi.$$
We may apply Lemma \ref{lemma6} with accuracy $\epsilon'=\epsilon\sqrt{\gamma/(50(L_1+2\gamma))}$ to bound the running time of each call of algorithm NAG by
\begin{equation*}
  \begin{aligned}
  \mathcal{O}\left(T_{\rm grad}\left(S+1+\sqrt{\frac{L_1+2\gamma}{\gamma}}\log\left(\frac{2^{S+2}(L_1+2\gamma)^{S+2}\Delta_{g_j}}{\gamma^{S+1}\epsilon'^2}\right)\right)\right).
  \end{aligned}
\end{equation*}
Using the fact that $\gamma\leq L_1$ and $\mathcal{O}(\Delta_{g_j})=\mathcal{O}(\Delta_\psi)$, we may ulteriorly bound the running time of each call by
$$\mathcal{O}\left(T_{\rm grad}\sqrt{\frac{L_1}{\gamma}}\log\left(\frac{L_1^{S+3}\Delta_\psi}{\gamma^{S+2}\epsilon^2}\right)\right).$$
The algorithm Semiconvex-NAG performs at most $j_*$ iterations, which yields the running time (\ref{time1}). Moreover, the inequality (\ref{e1}) is shown in \cite{ref13}.
\end{proof}

\subsection{Negative Curvature Exploitation}\label{section5.2}

The Negative Curvature Descent (NCD) method has been employed to develop deterministic or stochastic algorithms for nonconvex optimization with the benefits of escaping from non-degenerate saddle points and searching for a region where the objective function is semiconvex (equivalently, almost-convex) that enables accelerated gradient methods \cite{Liu-2018}. NCD is first proposed in \cite{Goldfarb-1980} and widely studied recently \cite{Jin-2018,Liu-2018,ref13}. In this subsection, basic and classic results, mainly from \cite{Goldfarb-1980,ref13}, are introduced for the convenience of the discussion in Section \ref{section5.3}.

\begin{lemma}\label{lemma7}
Let $K\in\mathbb{R}^{m\times n}$, $v\in\mathbb{R}^{mn}$ and  $\psi:\mathbb{R}^{m\times n}\longmapsto\mathbb{R}$ be $\mathcal{C}^2$-continuous function. Then, $\psi$ can be equivalently regarded as a function $\psi:\mathbb{R}^{mn}\longmapsto\mathbb{R}$. Moreover, we have
$$v^{\rm T}\nabla^2 \psi({\rm vec}(K))v=\nabla^2 \psi(K)[{\rm matrix}(v),{\rm matrix}(v)].$$
\end{lemma}
\begin{proof}
  The proof is easy, and we omit the proof here.
\end{proof}
\begin{breakablealgorithm}%[H]
    \caption{Negative Curvature Descent (NCD)}%算法标题
    \begin{algorithmic}[1]%一行一个标行号
    \State \textbf{Given} : Object function $\psi$, initial point $K_1$, probability accuracy $\delta$, smoothness parameter $L_2$, semi-convexity parameter $\alpha$, optimal gap $\Delta_\psi$
    \State \textbf{Result} : $K_j$
    \State Set $\delta'=\delta/(1+L_2^2\Delta_\psi/\alpha^3)$
    \For{$j=1,2,...$}
        \State Find a vector $v_j\in\mathbb{R}^{mn}$ such that $\Vert v_j\Vert =1$ and, with probability at least $1-\delta'$ by setting $\hat{v}_j={\rm matrix}(v_j)$,
	$$\lambda_1(\nabla^2 \psi({\rm vec}(K_j)))\geq \nabla^2 \psi(K_j)[\hat{v}_j,\hat{v}_j]-\alpha/2$$
	using leading eigenvector computation\
        \If{$\nabla^2 \psi(K_j)[\hat{v}_j,\hat{v}_j]\leq-\alpha/2$}
        \begin{equation*}
          \begin{aligned}
          K_{j+1}\leftarrow K_j-\frac{2\left|\nabla^2 \psi(K_j)[\hat{v}_j,\hat{v}_j]\right|}{L_2}
          {\rm sign}({\rm Tr}(\hat{v}_j^{\rm T}\nabla \psi(K_j)))\hat{v}_j
          \end{aligned}
        \end{equation*}
        \Else
            \State \Return {$K_j$}
        \EndIf
    \EndFor
    \end{algorithmic}
\end{breakablealgorithm}

The fifth step of the preceding routine is an additive $\epsilon$-approximate smallest eigenvector-seeking ($\epsilon$-ASES) problem, which appears forlorn and can be reformulated based on Lemma \ref{lemma7} as finding vector $v$ for a function $\rho:\mathbb{R}^d\longmapsto\mathbb{R}$ such that $v^{\rm T}\nabla^2 \rho(x)v\leq\lambda_1(\nabla ^2\rho(x))+\epsilon$. Nonetheless, if $\rho$ is additionally $L$-smooth, an $\epsilon$-ASES problem can be converted into a relative $\epsilon$-approximate leading eigenvector of a positive semidefinite matrix, which has been extensively studied. The following lemma gives a complexity analysis of the ASES problem.
\begin{lemma}[\cite{ref13}]\label{lemma8}
Let $\alpha\in(0,L]$. In the setting of the previous paragraph, there exists an algorithm that given $x\in\mathbb{R}^d$ computes, with probability at least $1-\delta$, an additive $\alpha$-approximate smallest eigenvector $v$ of $\nabla^2\rho(x)$ in time $\mathcal{O}(T_{\rm grad}\sqrt{\frac{L}{\alpha}}\log(\frac{d}{\delta}))$.
\end{lemma}
\begin{theorem}\label{thm8}
  Let $\alpha\in(0,L_1]$, $K_1\in\mathbb{R}^{m\times n}$ and function $\psi:\mathbb{R}^{m\times n}\longmapsto\mathbb{R}$ be $L_1$-smooth with $L_2$-Lipschitz continuous Hessian. If we call NCD$(\psi,K_1,\delta,L_2,\alpha,\Delta_\psi)$, then the algorithm terminates at iteration $j$ for some
  \begin{equation}\label{NCD1}
		j\leq 1+\frac{12L_2^2(\psi(K_1)-\psi(K_j))}{\alpha^3}\leq 1+\frac{12L_2^2\Delta_\psi}{\alpha^3},
  \end{equation}
  and with probability at least $1-\delta$,
  \begin{equation}\label{NCD2}
		\lambda_1(\nabla^2 \psi({\rm vec}(K_j)))\geq -\alpha.
  \end{equation}
  Furthermore, each iteration requires time at most
  \begin{equation}\label{NCD3}
		\mathcal{O}\left({ T_{\rm grad}}\sqrt{\frac{L_1}{\alpha}}\log\left(\frac{mn}{\delta}\left(1+\frac{12L_2^2\Delta_\psi}{\alpha^3}\right)\right)\right).
  \end{equation}
\end{theorem}
\begin{proof}
The proof's specifics can be found in Lemma 4.2 of \cite{ref13}, and for the sake of completeness, we only present the proof's outline here. Assume that the method has not terminated at iteration $k$. By denoting the step size used at iteration $k$ as $\eta_k$, i.e.,
$$\eta_k=\frac{2\left.|\nabla^2 \psi(K_k)[\hat{v}_k,\hat{v}_k]\right|}{L_2}\\
          {\rm sign}({\rm Tr}(\hat{v}_k^{\rm T}\nabla \psi(K_k))),$$
we have the following inequality according to the $L_2$-Lipschitz continuity of Hessian
\begin{equation*}
\begin{aligned}
&\left\vert \psi(K_k-\eta_k\hat{v}_k)-\psi(K_k)+\eta_k{\rm Tr}\left(\hat{v}_k^{\rm T}\nabla\psi(K_k)\right)-\frac{1}{2}\eta_k^2\nabla^2\psi(K_k)[\hat{v}_k,\hat{v}_k]\right\vert\leq\frac{L_2}{6}\Vert\eta_k\hat{v}_k\Vert_F^3.
\end{aligned}
\end{equation*}
We rearrange the above inequality to obtain
\begin{equation}\label{thm8result}
\psi(K_{k+1})-\psi(K_k)\leq-\frac{\alpha^3}{12L_2^2}.
\end{equation}
Telescoping the preceding equation for $k=1,2,\dots,j-1$, we conclude that at the final iteration
$$\Delta_\psi\geq\frac{\alpha^3}{12L_2^2}(j-1);$$
this gives the bound (\ref{NCD1}).

(\ref{NCD2}) and (\ref{NCD3}) are straightforward based on the definition of the algorithm NCD and Lemma \ref{lemma8}.
\end{proof}

\subsection{An Acceleration Framework of OLQR Problem}\label{section5.3}

We have introduced the two subroutines Semiconvex-NAG and NCD above, which only take into consideration the general nonconvex functions and ignore the unique structure of OLQR problem. In order to provide an accelerated framework for OLQR problem, we will  combine the aforementioned subroutines while taking advantage of specific properties of OLQR problem. In this subsection, we first review the fundamental properties of OLQR problem and provide a few notations.
\begin{figure}[htbp]
\centerline{\includegraphics[width=0.45\textwidth]{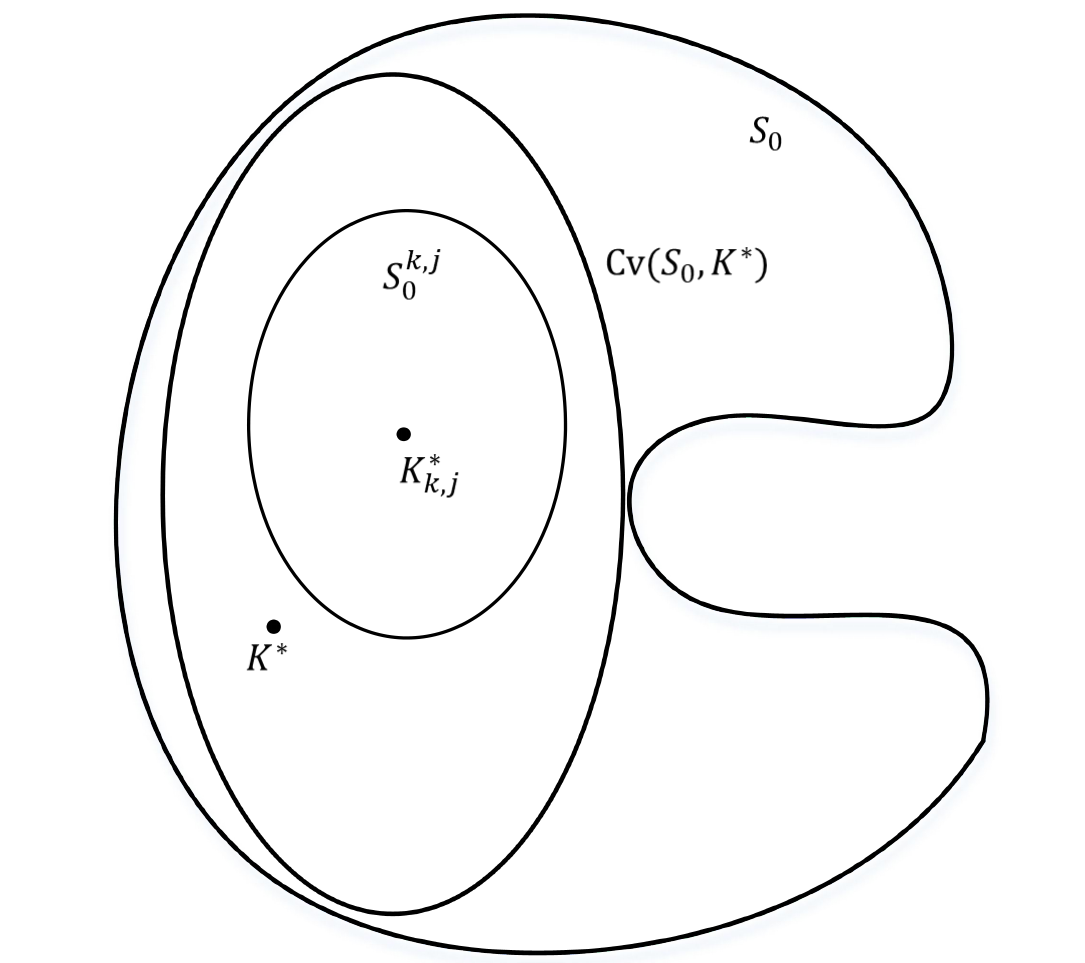}}
\caption{Structure of OLQR problem}
\label{fig2}
\end{figure}

Let $f(K)$ be OLQR cost of problem (\ref{eq7}). In Figure \ref{fig2}, $\mathcal{S}_0$ is given by the sublevel set $\mathcal{S}_0=\{K\in\mathcal{S}\colon f(K)\leq f(K_0)=\varrho\}$. Noticing that $\mathcal{S}_0$ may be not connected and without distinction, we refer to $\mathcal{S}_0$ as one of the sublevel set's connected components. $K^*\in\mathcal{S}_0$ is a  stationary point of $f(K)$; for any $K$, ${\rm Cv}(\mathcal{S}_0,K)$ is a convex set which is contained in $\mathcal{S}_0$ and contains $K$, at the meantime, we want it to be as large as feasible. The set $\mathcal{S}_0^{k,j}$ and point $K_{k,j}^*$ will be defined later. Hence, in ${\rm Cv}(\mathcal{S}_0,K^*)$, $f(K)$ is $L_1(\varrho)$-smooth and has $L_2(\varrho)$-Lipschitz continuous Hessian with
\begin{align*}
L_1(\varrho)&=\frac{2\varrho}{\lambda_1(Q)}(\lambda_n(R)\Vert C\Vert^2+\Vert B\Vert\Vert C\Vert_F\xi(\varrho)),\\
L_2(\varrho)&=2\Vert B\Vert\Vert C\Vert\frac{\varrho^2}{\lambda_1(Q)\lambda_1(\Sigma)}(2\kappa_3(\varrho)+\kappa_4(\varrho)),
\end{align*}
where $\xi(\varrho)$ is defined as
\begin{equation*}
      \begin{split}
         \xi(\varrho)=\frac{\sqrt{n}\varrho}{\lambda_1(\Sigma)}\left(\frac{\varrho\Vert B\Vert}{\lambda_1(\Sigma)\lambda_1(Q)}+\sqrt{\left(\frac{\varrho\Vert B\Vert}{\lambda_1(\Sigma)\lambda_1(Q)}\right)^2+\lambda_n(R)}\right),
      \end{split}
\end{equation*}
and $\kappa_3(\varrho)$, $\kappa_4(\varrho)$ are defined in Theorem \ref{LipschizHessian}. For a parameter $\alpha\geq 0$, we define the convex penalty function
\begin{equation*}
\rho_\alpha(K)=L_1(\varrho)\left[\Vert K\Vert_F-\frac{\alpha}{L_2(\varrho)}\right]_+^2,
\end{equation*}
where $[t]_+=\max\{t,0\}$. Furthermore, by adding the penalty function $\rho_\alpha$ to the performance criterion $f(K)$, we obtain
$$f_k(K)=f(K)+\rho_\alpha(K-\hat{K}_k),$$
where $\hat{K}_k$ will be defined later in pseudocode. The following lemma is an illustration of $f_k$'s enticing properties.
\begin{lemma}[\cite{ref13}]
Let $K_0\in{\rm Cv}(\mathcal{S}_0,K_0)$ such that, for any $E\in\mathbb{R}^{m\times r}$, $\nabla^2 f(K_0)[E,E]\geq -\alpha\Vert E\Vert_F^2$ for some $\alpha\geq 0$. Then the function $f_k(K)=f(K)+\rho_\alpha(K-K_0)$ is $3\alpha$-semiconvex and $3L_1(\varrho)$-smooth in ${\rm Cv}(\mathcal{S}_0,K_0)$.
\end{lemma}
\begin{proof}
As mentioned above, $f(K)$ is $L_1(\varrho)$-smooth and has $L_2(\varrho)$-Lipschitz continuous Hessian. Hence, this lemma is a direct corollary of Lemma 5.1 of \cite{ref13}.
\end{proof}
\begin{breakablealgorithm}%[H]
    \caption{Acceleration Framework of OLQR (A-OLQR)}%算法标题
    \begin{algorithmic}[1]%一行一个标行号
    \State \textbf{Given} : Performance criterion $f$, initial point $K_1$, accuracy $\epsilon$, smoothness parameter $L_1\geq L_1(\varrho)$, Lipschitz Hessian parameter $L_2\geq L_2(\varrho)$, semi-convexity parameter $\alpha$, optimal gap $\Delta_f$, probability accuracy $\delta$
    \State \textbf{Result} : $\hat{K}_k$
    \State Set $\Xi=\lceil 1+\Delta_f(12L_2^2/\alpha^3+\sqrt{10}L_2/(\alpha\epsilon))\rceil$ and $\delta''=\frac{\delta}{\Xi}$
    \For{$k=1,2,...$}
        \State $\hat{K}_k\leftarrow$NCD$(f,K_k,\delta'',L_2,\alpha,\Delta_f)$
        \If{$\Vert\nabla f(\hat{K}_k)\Vert<\epsilon$}
           \Return {$\hat{K}_k$}
        \EndIf
        \State Set $f_k(K)=f(K)+L_1([\Vert K-\hat{K}_k\Vert_F-\alpha/L_2]_+)^2$
        \State $K_{k+1}\leftarrow$Semiconvex-NAG$(f_k,\hat{K}_k,\epsilon/2,3L_1,3\alpha)$
    \EndFor
    \end{algorithmic}
    \label{alg4}
\end{breakablealgorithm}

Next, we will discuss some important assumptions and comments. The fundamental concept of the following discussion is to confine trajectory $K_k$ to the set $\mathcal{S}_0$, ensuring that the performance criterion $f(K)$ maintains the aforementioned smoothness properties so that Semiconvex-NAG and NCD are applicable.
Noticing that in Algorithm \ref{alg4}, every time the Semiconvex-NAG is executed, the Semiconvex-NAG's sub-algorithm named NAG requires no more than $j_*$ executions, where $j_*$ can be determined by the proof of Theorem \ref{thm7}. Let $K_{k,j}$ be the result of the $j$-th execution of NAG in the $k$-th execution of Semiconvex-NAG. Every calling of NAG is equivalent to optimizing the following $3\alpha$-strongly convex function
      \begin{equation*}
      \begin{aligned}
        f_{k,j}(K)=f(K)&+L_1([\Vert K-\hat{K}_k\Vert_F-\alpha/L_2]_+)^2+3\alpha\Vert K-K_{k,j-1}\Vert^2_F.
      \end{aligned}
      \end{equation*}

\begin{assumption}\label{core_assumption}
The following assumptions hold.
\begin{enumerate}
  \item Assuming that each time the sub-algorithm NCD is invoked, the resulting iteration sequence remains in set $\mathcal{S}_0$.
  \item For each $k,j$, it is assumed that the set ${\rm Cv}(\mathcal{S}_0,\hat{K}_k)$ contains the global minimizer of $f_{k,j}$ denoted as $K_{k,j}^*$.
  \item Let $\mathcal{S}_0^{k,j}$ shown in Figure \ref{fig2} be the sublevel set of $f_{k,j}$, i.e.,
        $$\mathcal{S}_0^{k,j}=\{K\colon f_{k,j}(K)\leq f_{k,j}(K_{k,j-1})\}.$$
        It is assumed that we have $\mathcal{S}_0^{k,j}\subseteq {\rm Cv}(\mathcal{S}_0,\hat{K}_k)$ for each $k,j$.
\end{enumerate}
\end{assumption}

\begin{remark}
The Assumption \ref{core_assumption} actually holds without loss of generality.
\begin{itemize}
  \item For Assumption \ref{core_assumption}.1, since for any $K\in\mathcal{S}_0$, $f(K)$ is
  $L_1(\varrho)$-smooth and has $L_2(\varrho)$-Lipschitz continuous Hessian on the set ${\rm Cv}(\mathcal{S}_0, K)$.
  In addition, according to (\ref{thm8result}), we have shown that the resulting iteration sequence of NCD is actually strictly decreasing sequence. Hence, this assumption actually holds true by choosing sufficiently big $L_2$.
  \item For Assumption \ref{core_assumption}.2, Since $f(K)$ is coercive and $\hat{K}_k$ is contained in the set ${\rm Cv}(\mathcal{S}_0,\hat{K}_k)$, this assumption actually holds true  by choosing sufficiently big $L_1$ and $L_2$.
  \item For Assumption \ref{core_assumption}.3, we can locate a point $\tilde{K}_{k,j-1}$ such that $\tilde{\mathcal{S}}_0^{k,j}=\{K\colon f_{k,j}(K)\leq f_{k,j}(\tilde{K}_{k,j-1})\}\subseteq{\rm Cv}(\mathcal{S}_0,\hat{K}_k)$ by employing vanilla gradient descent with the initial point $K_{k,j-1}$.
\end{itemize}

\end{remark}

In order to restrict the resulting iteration sequence to the set $\mathcal{S}_0^{k,j}$, the following restarting rule is introduced for each execution of NAG:
\begin{equation}\label{rr1}
    \begin{aligned}
        &\text{if }f_{k,j}(K_{p+1}^{k,j})\geq f_{k,j}(K_{k,j-1}),\text{ restart NAG}(f_{k,j},K_p^{k,j},\epsilon' ,3L_1,3\alpha),
    \end{aligned}
\end{equation}
where $K_p^{k,j}$ represents the outcome of executing the NAG algorithm with the function $f_{k,j}$ after $p$ steps and $\epsilon'$ is given by
$$\epsilon'=\frac{\epsilon}{2}\sqrt{\frac{\alpha}{50(L_1+2\alpha)}}.$$

\begin{assumption}\label{restart_assumption}
Assume that the restarting rule given in (\ref{rr1}) will be executed no more than $S$ times for each $k,j$.
\end{assumption}

\begin{remark}
For fixed $k,j$, $f_{k,j}$ is a strongly convex function, and NAG discrete-time algorithm (see Algorithm \ref{alg1})
\begin{equation}\label{NAG}
\begin{aligned}
	   y_{j+1}&=K_j-\frac{1}{L}\nabla f_{k,j}(K_j)\\ K_{j+1}&=\left(1+\frac{\sqrt{\kappa}-1}{\sqrt{\kappa}+1}\right)y_{j+1}-\frac{\sqrt{\kappa}-1}{\sqrt{\kappa}+1}y_j
\end{aligned}
\end{equation}
is equivalent to the discrete-time algorithm (\ref{d1}) with
$$d=\frac{1}{\sqrt{\kappa}+1}\sqrt{L},~\beta=\frac{\sqrt{\kappa}-1}{\sqrt{\kappa}+1}\sqrt{L},$$
where $L,\kappa$ are smoothness and condition number of $f_{k,j}$ respectively \cite{ref11}.
Moreover, notice
\begin{align*}
  -2d\Vert p\Vert^2&\geq -2d\Vert p\Vert^2-\langle p, \nabla f_{k,j}(q+\beta p)-\nabla f_{k,j}(q)\rangle\\
  &\geq-2d\Vert p\Vert^2-\frac{1}{3\alpha\beta}\Vert\nabla f_{k,j}(q+\beta p)-\nabla f_{k,j}(q) \Vert^2\\
  &\geq-\left(2d+\frac{\beta L^2}{3\alpha}\right)\Vert p\Vert^2,
\end{align*}
where the inequality holds by the basic properties of function class $\mathcal{F}^{1,1}_{3\alpha, L}$ (see Section 2.1.3 of \cite{ref5} for details). Hence, the assumption of Proposition 10 of \cite{ref10} holds, and $K_{k,j}^*$ is an asymptotically stable equilibrium of discrete-time dynamics (\ref{NAG}).
By the same reason with Remark \ref{remark3}, we can assume that the restarting rule will be executed no more than $S_{k,j}$ times for fixed $k,j$. In addition, by the proof of Theorem \ref{thm7} and Theorem \ref{thm9}, $k,j$ can be upper bounded by $k^*,j_*$ respectively. Thus, we can choose $S=\max_{k,j}S_{k,j}$.
\end{remark}

\begin{theorem}\label{thm9}
Under Assumption \ref{core_assumption}, \ref{restart_assumption}, let $K_1\in\mathcal{S}_0$ satisfy $f(K_1)-\inf_z f(z)\leq \Delta_f$, and let $\delta\in(0,1)$,
$0<\epsilon\leq\min\{\Delta_f^{\frac{2}{3}}L_2^{\frac{1}{3}},\frac{L_1^2}{L_2}\}$ and set $\alpha=L_2^{\frac{1}{2}}\epsilon^{\frac{1}{2}}$.
Then, with probability at least $1-\delta$, A-OLQR$(f,K_1,\epsilon,L_1,L_2,\alpha,\Delta_f,\delta)$ returns a point $K$ that satisfies
$$\Vert \nabla f(K)\Vert_F\leq\epsilon,\quad\lambda_1(\nabla^2 f({\rm vec}(K)))\geq -2L_2^{1/2}\epsilon^{1/2}$$
in time
\begin{equation}\label{f1}
  \mathcal{O}\left({ T_{\rm grad}}\frac{\Delta_f L_1^{1/2}L_2^{1/4}}{\epsilon^{7/4}}\log\left(\frac{mr\sqrt{L_1\Delta_f}}{\delta\epsilon}\right)\right).
\end{equation}
\end{theorem}
\begin{proof}
  Lemma 5.2 of \cite{ref13} has shown that with probability $1-\delta$, the method A-OLQR$(K_1,f,\epsilon,L_1,L_2,\alpha,$ $\Delta_f,\delta)$ terminates after $k^*$ iterations with
  $\Vert \nabla f(\hat{K}_{k^*})\Vert_F\leq\epsilon$ and $\lambda_1(\nabla^2 f({\rm vec}(\hat{K}_{k^*})))\geq -2\alpha$, where
  $$k^*\leq 2+\Delta_f\left(\frac{12L_2^2}{\alpha^3}+\frac{\sqrt{10}L_2}{\alpha\epsilon}\right)\leq18\Delta_fL_2^{\frac{1}{2}}\epsilon^{-\frac{3}{2}}.$$
  The last inequality holds from our assumption $\epsilon\leq\Delta_fL_2^{\frac{1}{2}}$.

  According to Theorem \ref{thm8}, the cost of the $k$-th iteration of NCD is bounded by
  \begin{equation*}
    \begin{aligned}
      \mathcal{O}&\left({ T_{\rm grad}}\left(1+\frac{12L_2^2(f(K_k)-f(\hat{K}_k))}{\alpha^3}\right)\sqrt{\frac{L_1}{\alpha}}\log\left(\frac{mr}{\delta''}\left(1+\frac{12L_2^2\Delta_f}{\alpha^3}\right)\right)\right).
    \end{aligned}
  \end{equation*}
  By the definition of $\delta''$, it is obvious that
  \begin{align*}
    \frac{mr}{\delta''}\left(1+\frac{12L_2^2\Delta_f}{\alpha^3}\right)\leq\frac{234mr}{\delta}\Delta_f^2L_2\epsilon^{-3}\leq\frac{234mr}{\delta}\Delta_f^2L_1^2\epsilon^{-4},
  \end{align*}
  with the last inequality following from $\epsilon\leq\frac{L_1^2}{L_2}$. Based on (\ref{NCD1}), we have $f(K_{k+1})\leq f(\hat{K}_k)$ and deduce that the cost of the $k$-th iteration of NCD is also bounded by
  \begin{equation}\label{jj}
    \begin{aligned}
      \mathcal{O}&\left({ T_{\rm grad}}\left(\frac{L_1^{\frac{1}{2}}}{L_2^{\frac{1}{4}}\epsilon^{\frac{1}{4}}}+\frac{L_1^{\frac{1}{2}}L_2^\frac{1}{4}}{\epsilon^{\frac{7}{4}}}(f(K_k)-f(K_{k+1}))\right)\log\left(\frac{mr}{\delta}\frac{\sqrt{L_1\Delta_f}}{\epsilon}\right)\right).
    \end{aligned}
  \end{equation}
  By suming this bound for $k=1,\dots,k^*$, we obtain the following bound of the total cost of NCD:
  \begin{equation*}
    \begin{aligned}
      \mathcal{O}\left({ T_{\rm grad}}\left(\frac{L_1^{\frac{1}{2}}}{L_2^{\frac{1}{4}}\epsilon^{\frac{1}{4}}}k^*+\frac{\Delta_fL_1^{\frac{1}{2}}L_2^\frac{1}{4}}{\epsilon^{\frac{7}{4}}}\right)\log\left(\frac{mr}{\delta}\frac{\sqrt{L_1\Delta_f}}{\epsilon}\right)\right).
    \end{aligned}
  \end{equation*}
  Combining with the estimation of $k^*$, we obtain the cost bound (\ref{f1}).

  Next, we will estimate the cost of invoking Semiconvex-NAG. In Theorem \ref{thm7}, we have shown that the cost of the $k$-th iteration of Semiconvex-NAG is bound by
  \begin{equation*}
    \begin{aligned}
      \mathcal{O}&\left({T_{\rm grad}}\left(\sqrt{\frac{3L_1}{\gamma}}+\frac{\sqrt{3\gamma L_1}}{\epsilon^2}(f_k(\hat{K}_k)-f_k(K_{k+1}))\right)\log\left(\frac{3^{S+3}L_1^{S+3}\Delta_f}{\gamma^{S+2}\epsilon^2}\right)\right).
      \end{aligned}
  \end{equation*}
  Substituting $\gamma=3\alpha=3\sqrt{L_2\epsilon}$, we can simplify the expression inside the logarithm by noticing
  \begin{align*}
    &\quad\log\left(\frac{3^{S+3}L_1^{S+3}\Delta_f}{\gamma^{S+2}\epsilon^2}\right)\\
    &=S\log\left(\frac{L_1^S}{L_2^{\frac{1}{2}}\epsilon^{\frac{1}{2}}}\right)+\log\left(\frac{3L_1^3\Delta_f}{L_2\epsilon^3}\right)\\
    &\leq S\log\left(L_1\Delta_f\epsilon^{-2}\right)+\log\left(3L_1^3\Delta_f^3\epsilon^{-6}\right)\\
    &=\mathcal{O}\left(\log\left(\frac{\sqrt{L_1\Delta_f}}{\epsilon}\right)\right),
  \end{align*}
  where the above inequality follows from $\epsilon\leq\Delta_f^{\frac{2}{3}}L_2^{\frac{1}{3}}$. Observing the fact $f_k(\hat{K}_k)-f_k(K_{k+1})\leq f(\hat{K}_k)-f(K_{k+1})\leq f(K_k)-f(K_{k+1})$,
  we can show that the cost of the $k$-th iteration of Semiconvex-NAG can also be bounded by (\ref{jj}). Hence for the same reason, we conclude that the total cost of Semiconvex-NAG can also be bounded by (\ref{f1}).
\end{proof}
\begin{remark}
  For a $L$-smooth and Lipschitz Hessian function $f(x)$, we call $x$ is a $(\epsilon,\delta)$-SSP (second-order stationary point) of $f(x)$ if $\Vert\nabla f(x) \Vert\leq\epsilon$ and $\nabla^2 f(x)\succeq -\delta I$. In general,
  such $(\epsilon,\delta)$-SSP provides a better approximate of local minima. Different from \cite{ref2}, our acceleration framework of OLQR problem provides the SSP guarantee.
\end{remark}

\section{Numerical Examples}

In this section, a few numerical examples are represented that  illustrate the theoretical results of this paper.

\begin{example}
Solve the SLQR problem (\ref{eq5}) with parameters
\begin{equation*}
  A_n=\begin{pmatrix}
        0 & 1 & 0 & \cdots & 0 & 0 \\
        0 & 0 & 1 & \cdots & 0 & 0 \\
        0 & 0 & 0 & \cdots & 0 & 0 \\
        \vdots & \vdots & \vdots& \ddots & \vdots & \vdots \\
        0 & 0 & 0 & \cdots & 0 & 1 \\
        0 & 0 & 0 & \cdots & 0 & 0
      \end{pmatrix}_{n\times n},
  B_n=\begin{pmatrix}
        0 \\
        0 \\
        0 \\
        \vdots \\
        0 \\
        1
      \end{pmatrix}_{n\times 1},
\end{equation*}
$R=1$ and $Q=\Sigma=I_n$.
\end{example}

\textit{Solution}. Denoting the feedback gain as $K=[k_1,k_2,\dots,k_n]\in\mathbb{R}^{1\times n}$, it is clear that
\begin{equation*}
	\det(\lambda I-(A_n-B_nK))=\lambda^n+k_n\lambda^{n-1}+\dots+k_2\lambda+k_1.
\end{equation*}
Then,
\begin{equation*}
	\mathcal{S}=\{K\in\mathbb{R}^{1\times n}\colon\Delta_k(K)>0, k=1,\dots,n\},
\end{equation*}
where $\Delta_k(K)$ denotes the $k$-order leading principal minors of matrix
\begin{equation*}
	H(K)=\begin{pmatrix}
		k_n & k_{n-2} & k_{n-4} & \cdots & 0 & 0 \\
		1 & k_{n-1} & k_{n-3} & \cdots & 0 & 0 \\
		0 & k_{n} & k_{n-2} & \cdots & 0 & 0 \\
		\vdots & \vdots & \vdots & \ddots & \vdots & \vdots \\
		0 & 0 & 0 & \cdots & k_2 & 0 \\
		0 & 0 & 0 & \cdots & k_3 & k_1
	\end{pmatrix},
\end{equation*}
called Hurwitz matrix. Obviously, $\mathcal{S}$ is not convex with non-smooth boundary, so the standard Nesterov-type method of \cite{ref5} cannot be applied directly.
For $n=3$, we have
\begin{equation*}
	\mathcal{S}=\{K\in\mathbb{R}^{1\times 3}\colon k_1>0,k_2k_3>k_1\}.
\end{equation*}
Then, we choose the initial stabilizing gain as $K_0=[5,100,15]$, and perform the gradient descent (GD) of \cite{ref2} and Nesterov-type method (\ref{d1}) of this paper, respectively. The numerical result is shown in Figure \ref{fig3}.  \hfill $\square$
\begin{figure}[htbp]
	\centerline{\includegraphics[width=0.56\textwidth]{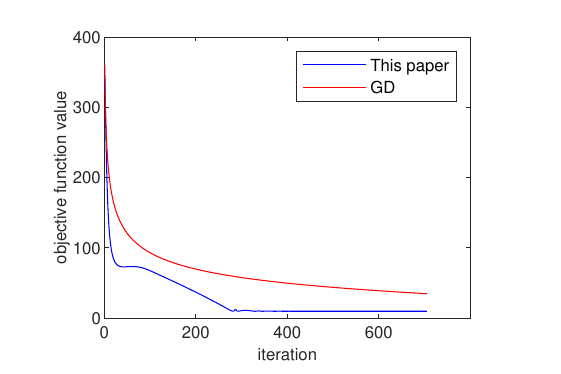}}
	\caption{The result for $n=3$ with $K_0=[5,100,15]$}
	\label{fig3}
\end{figure}

We can observe that the Nesterov-type method (\ref{d1}) converges faster than GD. To emphasize, our method is precisely the Heavy-ball scheme, which results in an accelerated convergence rate for an adequately small step-size.
Unlike Nesterov's original scheme introduced in \cite{ref5}, a gradient evaluation at $K_k+\beta p_k$ is not required, and typically only minor steps are permitted.
Regretfully, due to the smaller step-size selection compared to NAG or GD for ensuring the convergence of Heavy-ball method, our method does not exhibit an accelerated phenomenon or even appears worse performance in small-scale problems and in the case where the initial iteration point is very closed to the optimal feedback gain; see Figure \ref{fig4}.

\begin{example}\label{exp2}
Consider the SLQR problem given in Example 1 with $n=3$ and $K_0=[1,2,2]$, which is very near to the optimal feedback gain. The numerical result is shown in Figure \ref{fig4}.
\begin{figure}[htbp]
\centerline{\includegraphics[width=0.56\textwidth]{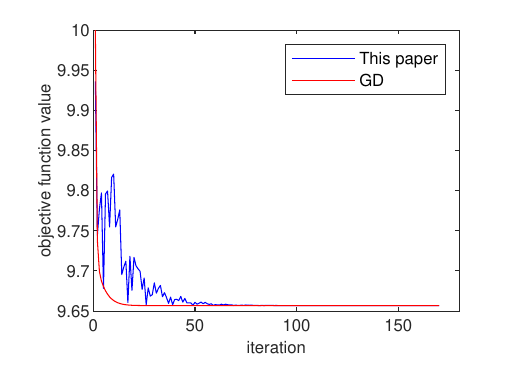}}
\caption{The result for $n=3$ with $K_0=[1,2,2]$}
\label{fig4}
\end{figure}
\end{example}

Intriguingly, when contemplating a problem of medium scale, the performance of our algorithm improves significantly after transient iterations; see Figure \ref{fig5}.

\begin{example}
Consider the SLQR problem given in Example 1 with $n=10$ and $$K_0=[1,10,45,120,210,252,210,120,45,10].$$
The numerical result is shown in Figure \ref{fig5}.
\begin{figure}[htbp]
\centerline{\includegraphics[width=0.56\textwidth]{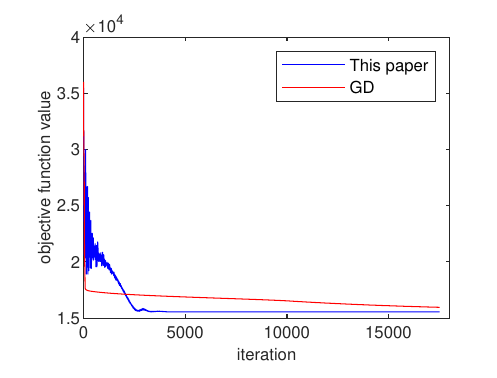}}
\caption{The result for $n=10$}
\label{fig5}
\end{figure}
\end{example}

\begin{example}

Consider a more general medium-size SLQR problem with parameters
\begin{align*}
  C&=I,\quad A=\frac{1}{n}{\rm rand}(n,n)-I ,\\
  B&={\rm ones}(n,m)+\frac{1}{2}{\rm rand}(n,m),\\
  Q&=Q_1Q_1^{\rm T},\quad Q_1={\rm rand}(n,n),\\
  R&=R_1R_1^{\rm T},\quad R_1={\rm rand}(m,m),
\end{align*}
where $n=10,m=3,{\rm ones}(n,m)$ is a $n\times m$ matrix with all entries equal to 0 and ${\rm rand}(n,m)$ is a $n\times m$ matrix with every entry generated from the uniform distribution on $[0,1]$. Choose the initial stabilizing gain as $K_0=0$, and we only show the result of two algorithms iterated after 30 times since $f(K_0)$ is too large in Figure \ref{fig6}.
\begin{figure}[htbp]
\centerline{\includegraphics[width=0.56\textwidth]{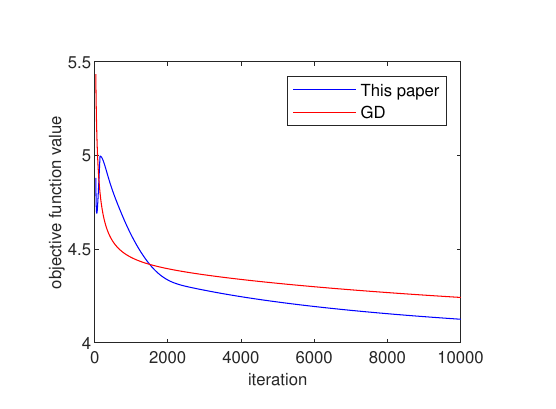}}
\caption{The result for a more general case}
\label{fig6}
\end{figure}

\end{example}

To sum up, for medium-scale problem or for the case when the initial point is far away from the optimal feedback gain, our algorithm (\ref{d1}) exhibits accelerated performance markedly. However, if considering a small-scale problem with a ``good'' initial point (initial point is near to the optimal feedback gain), our algorithm may perform worse (as shown in Example \ref{exp2}). This phenomenon stems from the momentum-based accelerated optimization algorithm itself \cite{z7}, rather than the SLQR problem introduced in this paper. Consequently, even if we consider a generally smooth and strongly convex function with ``good'' initial point, the standard NAG algorithm may perform less well than GD; this is because the iteration sequence of GD is decreasing around the global minima!

\section{Conclusion}

The results of this paper might be extended in several directions. First, a fixed-time stable gradient flow is introduced in  \cite{ref14}, which can converge in fixed time to the global minimizer of an objective function satisfying PL condition.
Can such a method be utilized to obtain a fixed-time or finite-time convergence algorithm for the SLQR problem in both continuous time and discrete time?
Second, more research effort should be devoted to the OLQR problem. For instance, how to find a strictly local minima? Can we design an accelerated framework which is one-procedure?
Third, can our results be extended to the Markovian jump linear system or indefinite LQR problem?

\section{Appendix}
\subsection{Proof of Theorem \ref{LipschizHessian}}\label{proof_of_hessian}
\begin{proof}
Since $\nabla^2f({\rm vec}(K))$ is symmetric, we have
\begin{align*}
\Vert\nabla^2f({\rm vec}(K))\Vert
&=\lambda_{\rm max}(\nabla^2f({\rm vec}(K)))\\
&=\sup_{\Vert \tilde{E}\Vert=1,\tilde{E}\in\mathbb{R}^{mn}}\left\vert\tilde{E}^{\rm T}\nabla^2f({\rm vec}(K))\tilde{E}\right\vert\\
&=\sup_{\Vert E\Vert_F=1,E\in\mathbb{R}^{m\times n}}\left\vert\nabla^2f(K)[E,E]\right\vert.
\end{align*}
The same procedure may be easily adapted to obtain
\begin{equation*}
	\Vert\nabla^2f({\rm vec}(K))-\nabla^2f({\rm vec}(K'))\Vert=\sup_{\Vert E \Vert_F=1}\left\vert\nabla^2 f(K)[E,E]-\nabla^2f(K')[E,E]\right\vert.
\end{equation*}
Denote $\Delta K=K'-K,X''_K[E]=\frac{{\rm d}^2}{{\rm d}\eta^2}\bigg|_{\eta=0}X_{K+\eta E}$ and define
$$g(\delta)=\nabla^2f(\bar{K})[E,E],$$
and one immediately has
\begin{align*}
	g'(\delta)&=\frac{\partial}{\partial \delta}{\rm Tr}\left(\frac{{\rm d}^2}{{\rm d}\eta^2}\bigg|_{\eta=0}X_{K+\delta(K'-K)+\eta E}\Sigma\right)\\
	&={\rm Tr}\left(\frac{\partial^3}{\partial \eta^2\partial \delta}\bigg|_{\eta=0}X_{K+\delta(K'-K)+\eta E}\Sigma\right).
\end{align*}
By the fundamental theorem of calculus, it follows that
\begin{align*}
&\Vert\nabla^2f({\rm vec}(K))-\nabla^2f({\rm vec}(K'))\Vert\\
	&=\sup_{\Vert E \Vert_F=1}\vert g(0)-g(1)\vert=\sup_{\Vert E \Vert_F=1}\left|\int_0^1g'(\delta){\rm d}\delta\right|\\
	&\leq\int_0^1\sup_{\Vert E \Vert_F=1}\vert g'(\delta)\vert{\rm d}\delta.
\end{align*}
Let us denote $X'_K[E]=\frac{{\rm d}}{{\rm d}\eta}\bigg|_{\eta=0}X_{K+\eta E}$, and based on the following Lyapunov equation given by (\ref{eq7})
\begin{equation}\label{L1}
A_{\bar{K}}^{\rm T}X_{\bar{K}}+X_{\bar{K}}A_{\bar{K}}+C^{\rm T}\bar{K}^{\rm T}R\bar{K}C+Q=0,
\end{equation}
we can get the Lyapunov equation
\begin{equation}\label{L2}
\begin{split}
A_{\bar{K}}^{\rm T}X'_{\bar{K}}[E]+X'_{\bar{K}}[E]A_{\bar{K}}-(BEC)^{\rm T}X_{\bar{K}}-X_{\bar{K}}BEC+C^{\rm T}E^{\rm T}R\bar{K}C+C^{\rm T}\bar{K}^{\rm T}REC=0
\end{split}
\end{equation}
by finding the directional derivative in direction $E$ at $\bar{K}$. By the same token, we can get the following Lyapunov equation
\begin{equation}\label{L3}
	\begin{split}
		A_{\bar{K}}^{\rm T}X''_{\bar{K}}[E]+X''_{\bar{K}}[E]A_{\bar{K}}-2(BEC)^{\rm T}X'_{\bar{K}}[E]
		-2X'_{\bar{K}}[E](BEC)+2C^{\rm T}E^{\rm T}REC=0.
	\end{split}
\end{equation}
Then,
\begin{equation*}
	\begin{split}
		A_{\bar{K}}^{\rm T}\frac{\partial^3}{\partial \eta^2\partial \delta}\bigg|_{\eta=0}X_{\bar{K}+\eta E}+\frac{\partial^3}{\partial \eta^2\partial \delta}\bigg|_{\eta=0}X_{\bar{K}+\eta E}A_{\bar{K}}
-2(BEC)^{\rm T}\frac{\partial X'_{\bar{K}}[E]}{\partial \delta}\\
-2\frac{\partial X'_{\bar{K}}[E]}{\partial \delta}(BEC)
-(B\Delta KC)^{\rm T}X''_{\bar{K}}[E]-X''_{\bar{K}}[E]B\Delta KC=0.
	\end{split}
\end{equation*}
Hence, we have
\begin{equation*}
	\frac{\partial^3}{\partial \eta^2\partial \delta}\bigg|_{\eta=0}X_{\bar{K}+\eta E}=\int_0^\infty e^{A_{\bar{K}}\tau}\tilde{S}e^{A_{\bar{K}}^{\rm T}\tau}{\rm d}\tau
\end{equation*}
by denoting
\begin{equation*}
  \tilde{S}=-2(BEC)^{\rm T}\frac{\partial X'_{\bar{K}}[E]}{\partial \delta}
		-2\frac{\partial X'_{\bar{K}}[E]}{\partial \delta}(BEC)-(B\Delta KC)^{\rm T}X''_{\bar{K}}[E]-X''_{\bar{K}}[E]B\Delta KC.
\end{equation*}
Based on the abovementioned fact, we can get that
\begin{align*}
	g'(\delta)&={\rm Tr}\left(\frac{\partial^3}{\partial \eta^2\partial \delta}\bigg|_{\eta=0}X_{\bar{K}+\eta E}\Sigma\right)\\
	&={\rm Tr}\left(\int_0^\infty e^{A_{\bar{K}}\tau}\tilde{S}e^{A_{\bar{K}}^{\rm T}\tau}{\rm d}\tau\Sigma\right)\\
	&={\rm Tr}\left(\tilde{S}\int_0^\infty e^{A_{\bar{K}}^{\rm T}\tau}\Sigma e^{A_{\bar{K}}\tau}{\rm d}\tau\right)\\
	&=-4{\rm Tr}\left((BEC)^{\rm T}\frac{\partial X'_{\bar{K}}[E]}{\partial \delta}Y_{\bar{K}}\right)-2{\rm Tr}\left((B\Delta KC)^{\rm T}X''_{\bar{K}}[E]Y_{\bar{K}}\right),
\end{align*}
and we then need to estimate the upper bound of items ${\rm Tr}(Y_{\bar{K}}),\frac{\partial X'_{\bar{K}}[E]}{\partial \delta}$ and $X''_{\bar{K}}[E]$ respectively.

First, from Lemma A.1 of \cite{ref2}, it is clear that
\begin{align*}
	f(\bar{K})={\rm Tr}((Q+C^{\rm T}\bar{K}^{\rm T}R\bar{K}C)Y_{\bar{K}})\geq \lambda_1(Q){\rm Tr}(Y_{\bar{K}}),
\end{align*}
so we have
\begin{equation*}
	{\rm Tr}(Y_{\bar{K}})\leq \frac{f(\bar{K})}{\lambda_1(Q)}\leq\frac{\alpha}{\lambda_1(Q)}.
\end{equation*}
%Here, we get the upper bound of ${\rm Tr}(Y_{\bar{K}})$, and

Second, we estimate the upper bound of $\frac{\partial X_{\bar{K}}}{\partial \delta}$.
Based on (\ref{L1}), we can obtain the following Lyapunov equation
\begin{equation*}
	\begin{split}
		A_{\bar{K}}^{\rm T}\frac{\partial X_{\bar{K}}}{\partial \delta}+\frac{\partial X_{\bar{K}}}{\partial \delta}A_{\bar{K}}-(B\Delta KC)^{\rm T}X_{\bar{K}}
		-X_{\bar{K}}(B\Delta KC)+C^{\rm T}\Delta K^{\rm T}R\bar{K}C+C^{\rm T}\bar{K}^{\rm T}R\Delta KC=0.
	\end{split}
\end{equation*}
Let us denote
\begin{align*}
S_1=-(B\Delta KC)^{\rm T}X_{\bar{K}}-X_{\bar{K}}(B\Delta KC)+C^{\rm T}\Delta K^{\rm T}R\bar{K}C+C^{\rm T}\bar{K}^{\rm T}R\Delta KC,
\end{align*}
and according to the Lemma  C.3 of \cite{ref2}, we have for any $K\in\mathcal{S}_{\alpha}$
\begin{equation*}
	\Vert K \Vert_F\leq \frac{2\Vert B\Vert \alpha}{\lambda_1(\Sigma)\lambda_1(R)}+\frac{\Vert A\Vert}{\Vert B\Vert}\doteq\zeta.
\end{equation*}
Combining with the fact that
\begin{equation*}
	\Vert X_{\bar{K}}\Vert \leq\frac{f(\bar{K})}{\lambda_1(\Sigma)}\leq\frac{\alpha}{\lambda_1(\Sigma)},
\end{equation*}
we can demonstrate that
\begin{align*}
	S_1&\preceq\left(2\Vert B\Vert\Vert \Delta K\Vert\Vert C\Vert\frac{\alpha}{\lambda_1(\Sigma)}+2\Vert C\Vert^2\Vert R\Vert \zeta \Vert \Delta K\Vert\right)I\\	
	&\preceq \frac{2}{\lambda_1(Q)}\left(\Vert B\Vert\Vert C\Vert\frac{\alpha}{\lambda_1(\Sigma)}+\Vert C\Vert^2\Vert R\Vert \zeta\right)\Vert \Delta K\Vert Q\\
	&\doteq \kappa_1\Vert\Delta K\Vert Q,
\end{align*}
where $\kappa_1$ is given by
$$\kappa_1=\frac{2}{\lambda_1(Q)}\left(\Vert B\Vert\Vert C\Vert\frac{\alpha}{\lambda_1(\Sigma)}+\Vert C\Vert^2\Vert R\Vert \zeta\right).$$
Hence, we can show that
\begin{align*}
	\frac{\partial X_{\bar{K}}}{\partial \delta}&=\int_0^\infty e^{A_{\bar{K}}\tau}S_1e^{A_{\bar{K}}^{\rm T}\tau}{\rm d}\tau\\
	&\preceq \kappa_1 \Vert\Delta K\Vert\int_0^\infty e^{A_{\bar{K}}\tau}Qe^{A_{\bar{K}}^{\rm T}\tau}{\rm d}\tau	\\
	&\preceq \kappa_1 \Vert\Delta K\Vert\int_0^\infty e^{A_{\bar{K}}\tau}(Q+C^{\rm T}\bar{K}^{\rm T}R\bar{K}C)e^{A_{\bar{K}}^{\rm T}\tau}{\rm d}\tau	\\
	&=\kappa_1 \Vert\Delta K\Vert X_{\bar{K}},
\end{align*}
i.e.,
\begin{equation}\label{C1}
  \frac{\partial X_{\bar{K}}}{\partial \delta}\preceq\kappa_1 \Vert\Delta K\Vert X_{\bar{K}}.
\end{equation}

Next, we will give the upper bound of $X'_{\bar{K}}[E]$. We denote
\begin{align*}
  S_2=-(BEC)^{\rm T}X_{\bar{K}}-X_{\bar{K}}BEC+C^{\rm T}E^{\rm T}R\bar{K}C+C^{\rm T}\bar{K}^{\rm T}REC,
\end{align*}
and based on (\ref{L2}) we can also show that
\begin{align*}
	S_2  \preceq \left(2\Vert B\Vert\Vert C\Vert\frac{\alpha}{\lambda_1(\Sigma)}+2\Vert C\Vert^2\Vert R\Vert\zeta\right)\frac{Q}{\lambda_1(Q)}=\kappa_2 Q,
\end{align*}
where $\kappa_2$ is given by
$$\kappa_2=\frac{2}{\lambda_1(Q)}\left(\Vert B\Vert\Vert C\Vert\frac{\alpha}{\lambda_1(\Sigma)}+\Vert C\Vert^2\Vert R\Vert\zeta\right).$$
Hence by the same token, we have
\begin{equation}\label{C2}
	X'_{\bar{K}}[E]\preceq \kappa_2X_{\bar{K}}.
\end{equation}

Then, we will find the upper bound of $\frac{\partial X'_{\bar{K}}[E]}{\partial \delta}$. According to (\ref{L2}), we can get that
\begin{equation*}
	\begin{split}
		&A_{\bar{K}}^{\rm T}\frac{\partial X'_{\bar{K}}[E]}{\partial \delta}+\frac{\partial X'_{\bar{K}}[E]}{\partial \delta}A_{\bar{K}}-(BEC)^{\rm T}\frac{\partial X_{\bar{K}}}{\partial \delta}
-\frac{\partial X_{\bar{K}}}{\partial \delta}(BEC)\\
&-(B\Delta KC)^{\rm T}X'_{\bar{K}}[E]-X'_{\bar{K}}[E](B\Delta KC)
+C^{\rm T}E^{\rm T}R\Delta KC+C^{\rm T}\Delta K^{\rm T}EC=0.
	\end{split}
\end{equation*}
By denoting
\begin{equation*}
	\begin{split}
		S_3=-(BEC)^{\rm T}\frac{\partial X_{\bar{K}}}{\partial \delta}-\frac{\partial X_{\bar{K}}}{\partial \delta}(BEC)-(B\Delta KC)^{\rm T}X'_{\bar{K}}[E]\\
		-X'_{\bar{K}}[E](B\Delta KC)+C^{\rm T}E^{\rm T}R\Delta KC+C^{\rm T}\Delta K^{\rm T}EC,
	\end{split}
\end{equation*}
we can prove that
\begin{align*}
	S_3&\preceq 2\left(\Vert B\Vert\Vert C\Vert\left\Vert\frac{\partial X_{\bar{K}}}{\partial \delta}\right\Vert+\Vert B\Vert\Vert C\Vert\Vert X'_{\bar{K}}[E]\Vert\Vert \Delta K\Vert+\Vert C\Vert^2\Vert R\Vert\Vert\Delta K\Vert\right)I\\
	&\preceq \frac{2}{\lambda_1(Q)}\left(\kappa_1\Vert B\Vert\Vert C\Vert\frac{\alpha}{\lambda_1(\Sigma)}+\kappa_2\Vert B\Vert\Vert C\Vert\frac{\alpha}{\lambda_1(\Sigma)}+\Vert C\Vert^2\Vert R\Vert\right)\Vert\Delta K\Vert Q\\
	&=\kappa_3\Vert\Delta K\Vert Q,
\end{align*}
where
\begin{align*}
	\kappa_3=&\frac{2}{\lambda_1(Q)}\left(\kappa_1\Vert B\Vert\Vert C\Vert\frac{\alpha}{\lambda_1(\Sigma)}+\kappa_2\Vert B\Vert\Vert C\Vert\frac{\alpha}{\lambda_1(\Sigma)}+\Vert C\Vert^2\Vert R\Vert\right).
\end{align*}
Hence,
\begin{equation}\label{C3}
	\frac{\partial X'_{\bar{K}}[E]}{\partial \delta}\preceq\kappa_3\Vert\Delta K\Vert X_{\bar{K}}.
\end{equation}

Besides, we will prove that $X''_{\bar{K}}[E]\preceq\kappa_4X_{\bar{K}}$. Since (\ref{L3}) holds, we can show that
\begin{align*}
	S_4&=-2(BEC)^{\rm T}X'_{\bar{K}}[E]-2X'_{\bar{K}}[E](BEC)+2C^{\rm T}E^{\rm T}REC\\
	&\preceq \frac{2}{\lambda_1(Q)}\left(2\kappa_2\Vert B\Vert\Vert C\Vert\frac{\alpha}{\lambda_1(\Sigma)}+\Vert C\Vert^2\Vert R\Vert\right)Q\\
	&=\kappa_4 Q,
\end{align*}
where
\begin{equation*}
	\kappa_4=\frac{2}{\lambda_1(Q)}\left(2\kappa_2\Vert B\Vert\Vert C\Vert\frac{\alpha}{\lambda_1(\Sigma)}+\Vert C\Vert^2\Vert R\Vert\right).
\end{equation*}
Thus we can prove that
\begin{equation}\label{C4}
	X''_{\bar{K}}[E]\preceq\kappa_4X_{\bar{K}}.
\end{equation}

Finally, combining with (\ref{C1}), (\ref{C2}), (\ref{C3}), and (\ref{C4}), we will estimate the upper bound of $\vert g'(\delta)\vert$. Specifically,
\begin{align*}
	\vert g'(\delta)\vert&=\left\vert 4{\rm Tr}\left((BEC)^{\rm T}\frac{\partial X'_{\bar{K}}[E]}{\partial \delta}Y_{\bar{K}}\right)+2{\rm Tr}\left((B\Delta KC)^{\rm T}X''_{\bar{K}}[E]Y_{\bar{K}}\right)\right\vert\\
	&\leq 4\Vert B\Vert\Vert C\Vert \left\Vert \frac{\partial X'_{\bar{K}}[E]}{\partial \delta}\right\Vert{\rm Tr}(Y_{\bar{K}})+2\Vert B\Vert\Vert C\Vert \Vert \Delta K\Vert \Vert X''_{\bar{K}}[E]\Vert {\rm Tr}(Y_{\bar{K}})\\
	&\leq2\Vert B\Vert\Vert C\Vert\frac{\alpha^2}{\lambda_1(Q)\lambda_1(\Sigma)}(2\kappa_3+\kappa_4)\Vert \Delta K\Vert.
\end{align*}
Hence,
\begin{align*}
	&\quad\Vert\nabla^2J({\rm vec}(K))-\nabla^2J({\rm vec}(K'))\Vert\\
	&\leq\int_0^1\sup_{\Vert E \Vert_F=1}|g'(\delta)|{\rm d}\delta=M \Vert K-K'\Vert,
\end{align*}
where
\begin{equation*}
	M=2\Vert B\Vert\Vert C\Vert\frac{\alpha^2}{\lambda_1(Q)\lambda_1(\Sigma)}(2\kappa_3+\kappa_4).
\end{equation*}
Hence, this completes the proof.
\end{proof}

\vspace{12pt}
\end{document}